\documentclass[11pt,a4paper]{amsart}

\usepackage{verbatim}
\usepackage[T1]{fontenc}
\usepackage{graphicx}
\usepackage{enumerate}
\usepackage{amsmath,amsfonts,amssymb}
\usepackage{color}
\usepackage{cite}
\definecolor{citation}{rgb}{0.2,0.58,0.2} 
\definecolor{formula}{rgb}{0.1,0.2,0.6}
\definecolor{url}{rgb}{0.3,0,0.5}
\usepackage{pgf,tikz}
\usepackage{mathrsfs}
\usepackage{fancyhdr}

\usepackage{dsfont}
 
\newcommand{\reqnomode}{\tagsleft@false}

\usepackage[colorlinks,pdfpagelabels,pdfstartview = FitH,bookmarksopen = true,bookmarksnumbered = true,urlcolor=url,linkcolor = formula,plainpages = false,hypertexnames = false,citecolor = citation] {hyperref}

\def\dys{\displaystyle}
\def\vs{\vspace{1mm}}
\def\dxy{\,{\rm d}x{\rm d}y}
\def\dx{\,{\rm d}x}
\def\dy{\,{\rm d}y}

\allowdisplaybreaks
\makeatletter
\DeclareRobustCommand*{\bfseries}{%
  \not@math@alphabet\bfseries\mathbf
  \fontseries\bfdefault\selectfont
  \boldmath
}

\DeclareMathOperator*{\osc}{osc}

\makeatother

\newlength{\defbaselineskip}
\newcommand{\setlinespacing}[1]
           {\setlength{\baselineskip}{#1 \defbaselineskip}}

\newtheorem{theorem}{Theorem}

\newtheorem{definition}{Definition}
\newtheorem{remark}{Remark}

\newtheorem{lemma}{Lemma}
\newtheorem{proposition}{Proposition}

\numberwithin{equation}{section}
\allowdisplaybreaks[4]

\newcommand{\rr}{\varrho}
\newcommand{\snr}[1]{\lvert #1\rvert}
\newcommand{\nr}[1]{\lVert #1 \rVert}
\newcommand{\rif}[1]{(\ref{#1})}
\newcommand{\stackleq}[1]{\stackrel{\rif{#1}}{ \leq}}

\title[H\"older regularity for nonlocal double phase equations]{H\"older regularity \\ for 
  nonlocal double phase equations}

\author{Cristiana De Filippis}  \address{Cristiana De Filippis\\Mathematical Institute, University of Oxford\\ Andrew Wiles Building, Radcliffe Observatory Quarter, Woodstock Road, Oxford, OX26GG, Oxford, United Kingdom} \email{Cristiana.DeFilippis@maths.ox.ac.uk}
\author{Giampiero Palatucci}  \address{Giampiero Palatucci\\Dipartimento di Scienze Matematiche, Fisiche e Informatiche, Universit\`a di Parma\\ Parco Area delle Scienze 53/a, Campus, 43124 Parma, Italy} \email{giampiero.palatucci@unipr.it}

\begin{document}

\subjclass[2010]{Primary 35D10, 35B45;
Secondary 35B05, 35R05, 47G20, 60J75\vspace{1mm}}

\keywords{Quasilinear nonlocal operators, fractional Sobolev spaces, viscosity solutions, double phase functionals, H\"older continuity\vspace{1mm}}

\thanks{{\it Aknowledgements.}\ \, The first author has been supported by the Engineering and Physical Sciences Research Council (EPSRC): CDT Grant Ref. EP/L015811/1. 
%Moreover she gratefully acknowledges the hospitality of the University of Parma where this paper was partially carried out during her visit.
The authors would like to thank the referees for their useful suggestions, which allowed to improve the manuscript.
\vspace{1mm}}

\maketitle

\begin{abstract}
We prove some regularity estimates for viscosity solutions to a class of possible degenerate and singular  integro-differential equations whose leading operator switches between two different types of fractional elliptic phases, according to the zero set of a modulating coefficient $a=a(\cdot,\cdot)$. The model case is driven by the following nonlocal double phase operator,
$$
\int \!\frac{|u(x)-u(y)|^{p-2}(u(x)-u(y))}{|x-y|^{n+sp}}\!\dy 
+ \int \!a(x,y)\frac{|u(x)-u(y)|^{q-2}(u(x)-u(y))}{|x-y|^{n+tq}}\!\dy,
$$
where $q\geq p$ and $a(\cdot,\cdot)\geqq 0$.
Our results do also apply for inhomogeneous equations, for very general classes of measurable kernels. 
By simply assuming the boundedness of the modulating coefficient, we are able to prove that the solutions are H\"older continuous, whereas similar sharp results for the classical local case 
do~require~$a$~to be H\"older continuous.  
 To our knowledge, this is the first (regularity) result for nonlocal double phase problems.
\end{abstract}

\vspace{3mm}

{\small \tableofcontents}

\setlinespacing{1.08}

\newcommand{\R}{\mathds{R}^n}

\vspace{-19cm}
\begin{center}
 \rule{11.9cm}{0.5pt}\\[-0.1cm] 
{\sc {\small To appear in}\, {\it J.~Differential~Equations}}
\\[-0.25cm] \rule{11.9cm}{0.5pt}
\end{center}

\newpage

\section{Nonlocal double phase problems}
We deal with nonlocal {\it double phase equations\,}; that is, a class of, possible singular and degenerate, integro-differential equations whose leading operator switches between two different fractional elliptic phases according to the zero set of the {\it modulating coefficient} $a=a(\cdot,\cdot)$. These equations  are indeed driven by the following nonlocal double phase operator,
\begin{eqnarray}\label{operatore} 
\qquad \  \mathcal{L}(u)\!\!\!\!\!&:=&\!\!\!\!\!{P.~\!V.}\!\dys\int_{\R} {|u(x)-u(y)|^{p-2}(u(x)-u(y))}K_{sp}(x,y)\dy 
\\*
 && +\ {P.~\!V.}\!\dys\int_{\R} a(x,y){|u(x)-u(y)|^{q-2}(u(x)-u(y))}K_{tq}(x,y)\dy, \ \ \ x\in\R, \nonumber
\end{eqnarray}
where the involved kernels $K_{sp}, K_{tq}:\R\times\R\to(0,\infty)$ are measurable functions of differentiability orders $s,t\in(0,1)$ and summability exponents $p,q\in (1,\infty)$, respectively. Here~$P.~\!V.$~stands for the principal value.
We immediately refer to Section~\ref{sec_preliminaries} for the precise assumptions on the involved quantities in the general framework we are considering. In order to simplify, one can just keep in mind the model case when the kernels $K_{sp}$ and $K_{tq}$ do coincide with the Gagliardo kernels $|x-y|^{-n-sp}$ and $|x-y|^{-n-tq}$, respectively; i.~\!e., the case when the corresponding operator $\mathcal{L}$ does reduce to a sum of a pure $p$-fractional Laplacian~$(-\Delta)^s_p$ and an integro-differential operator whose $(t,q)$-kernel is perturbated by the modulating coefficient $a(\cdot,\cdot)$. 
\vspace{2mm}

Such a case can be plainly seen as the nonlocal analog of the classical double phase problems, whose chief model is related to by the following energy functional,
\begin{equation}\label{locale}
\mathcal{F}(u):=\dys \int\big(|Du|^p+a(x)|Du|^q\big)\dx, \quad 1<p\leq q,
\end{equation}
naturally defined for Sobolev functions.   
The  functional
 $\mathcal{F}$   originally
arose in Homogenization Theory and it is related to the so-called Lavrentiev phenomenon; see for instance~\cite{Zhi86,Zhi95}. From a regularity point of view, even without the presence of the modulating coefficient~$a(\cdot)$, such  
 functional presents very interesting features, falling in the class of the non-uniformly elliptic ones  having $(p,q)$-growth conditions. Thus, it cannot be treated via the standard available regularity methods; we refer the reader to the pioneering work by Marcellini~\cite{Mar89,Mar91,Mar93,Mar96}, where the fundamentals of the $(p,q)$-regularity theory have been settled. One of the main points in the important $(p,q)$-theory is the lack of regularity results for more general functionals whose integrand depends on $x$ possibly in a non-smooth way. In this respect, in view of
the presence of the modulating coefficient, the functional~$\mathcal{F}$ in~\eqref{locale} is conceivably the prototype of the worst kind of interplay between the coefficient in~$x$ and the $(p,q)$-growth, since it
 clearly brings a change of ellipticity/growth  occurring on the set~$\{ a=0\}$. Let us consider the significant case when $q>p$: in the points where $a>0$ the functional~$\mathcal{F}$~reduces to a non-standard $(p,q)$-growth functional, which exhibits a $q$-growth in the gradient. On the contrary, in the points where $a=0$ the functional exhibits a $p$-growth in the gradient. This is the main feature of this class of functionals and it is basically the reason why they have been firstly introduced by~Zhikov in the aforementioned papers in order to describe the behavior of strongly anisotropic materials whose hardening properties drastically change with the point. Such an important phase-transition problem is thus described by the functional~$\mathcal{F}$, where the regulation of the mixture between two different materials, with $p$ and $q$ hardening, is modulated by the coefficient~$a(\cdot)$, which simultaneously brought new difficulties in the  corresponding regularity theory. Indeed, even basic regularity issues for these double phase problems have remained unsolved for several decades. The first result in this spirit was recently due to~Colombo and Mingione in~\cite{CM15}, where, amongst other achievements, they proved 
 {\it H\"older continuity for the weak solutions by assuming that the modulating coefficient $a(\cdot)$ is H\"older continuous as well}. This important result is also proven to be sharp both with respect to the H\"older continuity assumption on the modulating coefficient and with respect to the result obtained. See in particular Theorem~1.1 in~\cite{CM15}, and also Remarks~\ref{rem_lavri}-\ref{rem_limit} in forthcoming Section~\ref{sec_preliminaries} and Section~\ref{sec_clarification} for a  discussion about the restriction on the involved quantities.
\vspace{2mm}

 Starting from the work of Colombo and Mingione,  
  despite its relatively short history, double phase problems have already evolved into an elaborate theory with several connections to other branches; the literature is too wide to attempt any comprehensive treatment in a single paper. We refer, for instance, to~\cite{BCM15,BCM18,BO17,CM15b,DEF19,DFO18,OK18,PR18} and the references therein.
\vspace{2mm}

Let us come back to the equations we are dealing with. In the present paper, we consider the nonlocal version of the double phase problems described above. More in general, we also consider inhomogeneous equations with a given datum $f\in L_{\rm loc}^\infty$. Our main result reads as follows:
\begin{eqnarray}\label{nonlocale}
&&\text{\it Suppose that the modulating coefficient } a(\cdot,\cdot) \ \text{\it  is such that } 0\leqq a\leq M, \\
&&\text{\it then any bounded viscosity solution } u\  \text{\it  to } \, \mathcal{L}u=f \ \text{\it  is  locally H\"older continuous}.\label{sushi}
\end{eqnarray}
For the precise assumptions and statement, we refer to forthcoming Section~\ref{sec_preliminaries}, and in particular to Theorem~\ref{teo_holder} there. 

\vspace{2mm}

 Now, a few observations are in order:\vspace{0.2mm}

$\bullet$ First of all, in contrast with respect to the local case, we are able to prove that (viscosity) solutions to the fractional analog of double phase equations are H\"older continuous by simply assuming the boundedness of the modulating coefficients $a(\cdot,\cdot)$. In the {corresponding} local case, the  
{$C^{1,\beta}$-regularity} result by Colombo and Mingione
 is {crucially} related to the {absence} of the Lavrentiev phenomenon. {This is achieved provided that the ratio $q/p$ is suitably bounded, i.e.: $1\le q/p\le 1+\alpha/p$,}  where the $\alpha\in(0,1]$ is the H\"older exponent of $a(\cdot)$; see \cite[Proposition 3.6]{CM15b} and also the related preliminary result for what concerns the analysis of the Lavrentiev effect in~\cite{ELM04}. Analogously, the much weaker assumption $a\in L^\infty$ in~\eqref{nonlocale} may   
be interpreted as an ineffectiveness of the Lavrentiev phenomenon in the nonlocal framework, and a consequent assumption on the differentiability exponents~$(s,t)$ and summability ones~$(p,q)$ will appear in accordance with the local results.
 
$\bullet$ Second, the equation~in~\eqref{sushi} inherits both the difficulties newly arising from the double phase problems and those naturally arising from the nonlocal character of the involved fractional integro-differential operators. More than this, 
it is worth noticing that the fractional operators~$\mathcal{L}$ in~\eqref{operatore} present as well the typical issues given by their {\it nonlinear} growth behavior, and further efforts are needed due to the presence of merely measurable coefficients in the kernels~$K_{sp}, K_{tq}$. For this, some very important tools recently introduced in the nonlocal theory, as the celebrated Caffarelli-Silvestre $s$-harmonic extension~\cite{CS07}, and 
 other minor successful tricks, as for instance the pseudo-differential commutator compactness in~\cite{PP14} or the energy density estimates in~\cite{PSV13,CMP15,FSV15},
  seem not to be adaptable to the framework considered here.

$\bullet$  Third, to our knowledge, this is the very first regularity result for solutions to nonlocal double phase equations. Even in the very special case when both the differentiability orders and the summability exponents coincide; that is, when $s=t$ and $p=q$, no related results involving a modulating coefficient could be found in the literature; however, it is worth mentioning the fine H\"older estimates in the relevant paper~\cite{KRS14}, where the authors deal with a class of elliptic integro-differential operators with  kernels 
 satisfying lower bounds  on conic subsets, thus
  strongly directionally dependent.
 
 \vspace{1mm}
 
 For what concerns our approach to attack the problem, we extend that in the by-now classical work by Silvestre in~\cite{Sil06}, where the author provides a surprisingly clean and purely analytical proof of H\"older continuity for harmonic functions with respect to a class of integro-differential equations like the fractional Laplace with coefficients. The approach developed by Silvestre also includes the case of variable orders, and it has been proven to be very feasible to attack several problems in the recent nonlocal theory, even in the case of the $p$-fractional Laplace equation, as seen in the recent paper by Lindgren~\cite{Lin16}. Clearly, the mentioned approach cannot be plainly applied to the operator in~\eqref{sushi} because our class of operators lives in the nonstandard $(p,q)$-growth setting, and, even worst, we have to take care of the novelty given by the presence of the modulating coefficient~$a(\cdot,\cdot)$. 
 {In addition, it is worth mentioning that such non-uniform ellipticity together with the interplay of the two differentiability orders via the modulating coefficient~$a$ will  preclude the natural scaling properties of pure fractional Laplace operators. In this respect, the aforementioned proof of the H\"older regularity for solutions to the $p$-fractional Laplace equation via the approach by Silvestre is applicable with no substantial modifications to the inhomogeneous case with a bounded datum~$f$; see Remark~4.3 in~\cite{Sil06} and Lemma~1 in~\cite{Lin16}. On the contrary, the nonlocal double phase equations treated here will require further efforts. Precisely, in order to carefully accomodate the presence of the datum~$f$, we  need to take into account an appropriate analysis of the scaling effects on the double phase equations, which will influence the involved kernels  and the modulating coefficient as well. We refer in particular to the detailed computations about such a nontrivial extension in the forthcoming proofs of Proposition~\ref{prop_cumbersome} and Lemma~\ref{lem_positivity}; see also the appendix.}
 Finally, the wide range of integrability we are considering here will force us to work in three different ranges depending on the interaction between the exponents~$q\geq p$, which can vary from singular to degenerate cases. 

\vspace{2mm}

\mbox{Several questions naturally arise:\vspace{0.2mm} }
\\*
\indent $\bullet$  Firstly, it is worth remarking that the result presented here does apply to bounded solutions. It could be interesting to treat unbounded solutions, namely by truncation and dealing with the resulting error term as a right hand-side, in the same flavour of the papers~\cite{DKP14,DKP16}, where
 the local boundedness for $p$-fractional minimizers is proven, given that a precise quantity, the so-called ``nonlocal tail'', can be controlled; see in particular the local~vs.~nonlocal interpolation estimate in~\cite[Theorem~1.1]{DKP16}. Notice that the approach  in the aforementioned papers --  as well as in the local counterpart of the double phase operators~\cite{CM15,CM15b} -- is  in the spirit of De~Giorgi-Nash-Moser, being addressed to energy minimizers/weak solutions,
 whereas the approach presented here is more in the spirit of Krylov-Safonov. 
 
 $\bullet$ For this, a second natural question is whether or not, and under which assumptions on the structural quantities, the viscosity solutions to nonlocal double phase equations are indeed fractional harmonic functions and/or weak solutions, and vice versa. In this respect, let us observe that one cannot plainly applied the results for  $p$-fractional minimizers as obtained in the recent paper~\cite{KKP17} together with those in the forthcoming paper~\cite{KKL19}, whose proofs seem to be feasible only for a restrict class of kernels which cannot include a general modulating coefficient. 
 
 $\bullet$ Third, in the same spirit of the series of paper by Baroni, Colombo, and Mingione, one would expect higher differentiability and regularity results for the bounded solutions to nonlocal double phase equations. We refer to~\cite{BCM18} and in particular to the paper~\cite{CM15b}, where it has been established the local gradient H\"older continuity of minima provided a sharp balancing condition between the closeness of $p$ and $q$ and the regularity of $a(\cdot)$ is satisfied. For what concerns the nonlocal case, it could be useful to start from  the estimates obtained in the present paper, together with those in the very relevant results obtained for the fractional $p$-Laplace equation by Brasco, Lindgren, and Schikorra~\cite{BL17,BLS18}.
 
 $\bullet$ Also, again in clear accordance with the local counterpart \cite[Theorem 1.2]{CM15b}, one would expect self-improving properties of the solutions to~\eqref{sushi}. For this, one should extend the recent nonlocal Gehring-type theorems proven in~\cite{KMS15,Sch16}. 

$\bullet$ Finally, both in the local and in the nonlocal double phase theory, nothing is known about the regularity  for  solutions to  parabolic double phase equations.

\vspace{2mm}

{\it To summarize}.~The regularity result in the present paper seems to be the very first one for nonlocal double phase equations; i.~\!e., a wide class of equations led by fractional operators exhibiting non-standard growth conditions and non-uniform ellipticity properties, the latter according to the geometry of the level set of a modulating coefficient $a(\cdot,\cdot)$.
Precisely, we prove H\"older continuity for bounded viscosity solutions, by extending to the fractional framework a recent sharp result by Colombo and Mingione, with the substantial difference on the requirement of the modulating coefficient $a(\cdot,\cdot)$, which is here assumed to belong only to $L^\infty$.
We believe our estimates to be important in a forthcoming nonlocal theory of double phase operators.

\vspace{2mm}
\section{Setting of the problem and description of the main result}\label{sec_preliminaries}
In this section we set the problem we are dealing with, and we state our main result, by adding further considerations about the involved quantities and assumptions.
\vspace{1mm}

Consider the following inhomogeneous nonlocal double phase equation,
\begin{equation}\label{problema}
\dys \mathcal{L}u=f,
\end{equation}
where $f$ is bounded and the integro-differential operator~$\mathcal{L}$ is given by
\begin{eqnarray}\label{L}
\mathcal{L}u(x)\!\!&:=&\!\!{P.~\!V.}\!\dys\int_{\mathds{R}^{n}}\snr{u(x)-u(x+y)}^{p-2}(u(x)-u(x+y))K_{sp}(x,y)  \dy\nonumber \\*
&&\!\!+\ {P.~\!V.}\!\dys\int_{\mathds{R}^{n}}a(x,y)\snr{u(x)-u(x+y)}^{q-2}(u(x)-u(x+y))K_{tq}(x,y) \dy.
\end{eqnarray}
In the display above, the symbol ${P.~\!V.}$ stands for ``principal value''; in the rest of the paper, when not important, or clear from the context, we shall omit such a symbol. The measurable kernels $K_{sp}$ and $K_{tq}$ essentially behave like $(s,p)$ and~$(t,q)$-kernels, respectively; see~\cite{DPV12,Pal18} 
 for the basics on fractional Sobolev spaces. More precisely, there exists a positive constant~$\Lambda$ such that
\begin{equation}\label{K}
\begin{cases}
\ {\Lambda^{-1} \snr{y}^{-n-sp}}\le K_{sp}(x,y)\le {\Lambda}{\snr{y}^{-n-sp}},\\[1ex]
\ K_{sp}(x,y)=K_{sp}(x,-y),
\end{cases} \hspace{-2mm}\text{and}\ \ 
\begin{cases}
\ {\Lambda^{-1} \snr{y}^{-n-tq}}\le K_{tq}(x,y)\le {\Lambda}{\snr{y}^{-n-tq}},\\[1ex]
\ K_{tq}(x,y)=K_{tq}(x,-y).
\end{cases}
\end{equation}
The differentiability orders $s,t \in (0,1)$ and the summability exponents $p, q>1$ are required such that 
\begin{equation}\label{pq}
p>\frac{1}{1-{s}} \ \ \mbox{if} \ \ p< 2 ,\qquad q>\frac{1}{1-{t}},
\end{equation}
{and}
\begin{equation}\label{hp_spq}
1\, \le\, \frac{q}{p}\, \leq \, \min \left\{\frac{s}{t},1+s\right\};
\end{equation}
for motivations about the two requirements above we refer to Remarks~\ref{rem_lavri}--\ref{rem_limit} below.
Finally, the modulating coefficient~$a=a(\cdot,\cdot)$ is assumed to be measurable and such that
\begin{equation}\label{a}
0\le a(x,y)\le M \ \  \text{for a.~\!e. } (x,y)\in \mathds{R}^{n}\times \mathds{R}^{n}.
\end{equation}

In order to prove that viscosity solutions~$u$ could behave as classical solutions (see in particular forthcoming Proposition~\ref{prop_classic}), we will naturally require that the corresponding nonlocal double phase energy of~$u$ is finite; i.~\!e.,
\begin{flalign}\label{energiafinita}
\int_{\mathds{R}^{n}}\int_{\mathds{R}^{n}}\frac{\snr{u(x)-u(x+y)}^{p}}{\snr{y}^{n+sp}}+a(x,y)\frac{\snr{u(x)-u(x+y)}^{q}}{\snr{y}^{n+tq}} \dxy<\infty.
\end{flalign}

Our main result is that bounded viscosity solutions to~\eqref{problema} with $f$ bounded are locally H\"older continuous, as stated in Theorem~\ref{teo_holder} below. For the natural definition of viscosity solutions to nonlocal double phase equations, we refer to forthcoming Section~\ref{sec_viscosity}. 
\begin{theorem}\label{teo_holder}
Let the nonlocal double phase operator $\mathcal{L}$ be defined in \eqref{L}, under the assumptions~\eqref{K}--\eqref{energiafinita}, 
and let $f$ be  in $L^\infty(B_2)$.  
 If $u$  
 is a  bounded
  viscosity solution to 
\[
\mathcal{L}u=f \ \ \mbox{in} \ \ B_{2},
\]
then $u\in C^{0,\gamma}(B_{1})$ for some $\gamma=\gamma({\tt{data}})\in (0,1)$. 
\end{theorem}
In order to shorten the notation, in the statement above as well as in the rest of the paper, we use the symbol~\texttt{data} to express the natural dependency of the constants as follows 
\[
\texttt{data}\,:= \,n,p,q,s,t,M,\Lambda,\nr{u}_{L^{\infty}(\mathds{R}^{n})},\nr{f}_{L^{\infty}(B_{2})}.
\]

\begin{remark}\label{rem_lavri}{\rm As mentioned in the introduction, in the local case, the sharp $C^{1,\beta}$-regularity result by Colombo and Mingione is strictly related to the effect of the Lavrentiev phenomenon, which can be avoided in the ``a priori bounded'' case by assuming $1\le q/p\le 1+\alpha/p$, \, $\alpha$ being the H\"older exponent of $a(\cdot)$. In clear accordance, the range of validity of our result in~\eqref{hp_spq} is precisely given by $1\leq q/p\leq \min\{s/t,\, 1+s\}$, being informally $\alpha=0$ here.
Notice also that such a bound can be improved when considering homogeneous equations; i.~\!e., when $f\equiv 0$, as follows, $1\le q/p \le s/t$. As we will see in the following, this is strictly related to the scaling properties of the nonlocal double phase equations.
}
\end{remark}
\begin{remark}\label{revr}
\emph{Condition~\eqref{pq}$_{2}$ is due to the presence of the modulating coefficient $a(\cdot, \cdot)$ which, in view of \eqref{a}, is only bounded. In fact, concerning the integrability exponent~$p$, we notice that the assumption in~$\eqref{pq}_{1}$ applies only when $p<2$, while, if no extra regularity 
%than the one assumed in \eqref{a} is imposed 
is given on $a(\cdot, \cdot)$, the assumption in~$\eqref{pq}_{2}$ becomes necessary in order to deal with the $q$-part of the involved energy functional. On the other hand, if $a$ belongs to~$C^{0,\alpha}(\mathds{R}^{n})$ for some $\alpha\in (0,1]$, condition~$\eqref{pq}_{2}$ can be improved as follows
\begin{flalign*}
\begin{cases}
 \displaystyle q>\frac{1-\alpha}{1-t} \ \ &\mbox{if}\ \ q\ge 2,\\[0.5ex]
 \displaystyle q>\frac{1}{1-t} \ \ &\mbox{if} \ \ q<2;
\end{cases}
\end{flalign*}
see Lemma~\ref{revL2} in the appendix. This is consistent with the more complex structure of nonlocal double phase equations with respect to that of  pure $p$/$p(x)$-fractional equations. 
}
\end{remark}
\begin{remark}\label{rem_limit}
{\rm It is worth noticing that the estimates obtained in the present paper are not uniform with respect to the differentiability exponents $s$ and $t$ as they approach~1. This is consistent with the result in the local case. Otherwise, Theorem~\ref{teo_holder} would have implied a very general H\"older continuity {result} for solutions to the local double phase problem without requiring the H\"older continuity of the modulating coefficient~$a(\cdot,\cdot)$, in contrast with the sharpness of the results in~\cite{CM15,CM15b}.
For  H\"older estimates and some higher regularity result for related $(s,p)$-Laplace equations when $s$ approach $1$, we refer to~\cite{BCF12}. 
}
\end{remark}

\vspace{2mm}
\section{Viscosity solutions to nonlocal double phase problems}\label{sec_viscosity}

{Let us} provide the natural definition of nonlocal viscosity solutions to~\eqref{problemac}, by also showing their connection with the classical solutions. We have the following
\begin{definition}\label{def_viscosity}
Let $\Omega \subset \mathds{R}^{n}$ be an open subset and ${\mathcal{L}}$ be as {in} \eqref{Lc}, {under assumptions~\eqref{K}-\eqref{a}}. An upper semicontinuous function $u \in L_{\rm loc}^{\infty}(\Omega)$  
is a subsolution of ${\mathcal{L}}(\cdot)= C$ in $\Omega$, and we write 
\[
\text{\textquotedblleft} u \ \text{is such that} \ {\mathcal{L}}(u)\leq C \ \text{in} \ \Omega \ \text{in the viscosity sense\quotedblbase}
\] 
if the following statement holds: whenever $x_{0}\in \Omega$ and $\varphi\in C^{2}(B_{\rr}(x_{0}))$ for some $\rr >0$ are so that
\begin{flalign*}
\varphi(x_{0})=u(x_{0}), \quad \varphi(x)\ge u(x) \ \ \mbox{for all} \ x \in B_{\rr}(x_{0})\Subset \Omega,
\end{flalign*}
then we have ${\mathcal{L}}\varphi_{\rr}{(x_{0})}\le C$, where
\[
\varphi_{\rr}:=\begin{cases}
 \ \varphi  & \mbox{in \ } B_{\rr}(x_{0})\\
 \ u  & \mbox{in \ } \mathds{R}^{n}\setminus B_{\rr}(x_{0}).
\end{cases}
\]
A {\it viscosity supersolution} is defined in an analogous fashion, and a {\it viscosity solution} is a function which is both a subsolution and a supersolution.
\end{definition}

As customary, in the definition above, we denoted by $B_\rho=B_{\rho}(x_0)$ the ball of radius~$\rho$ centered in $x_0$. We will keep this notation throughout the rest of the paper.
\vspace{1mm}
In the proposition below, we  show that as soon as we can touch a viscosity subsolution to~\eqref{problemac} with a $C^{2}$-function, then it behaves as a classical subsolution. The proof will extend to the double phase problems a by-now classical approach, as firstly seen in~\cite{CS09} for fully nonlinear integro-differential operators.
\begin{proposition}\label{prop_classic}
Under assumptions \eqref{K}-\eqref{energiafinita}, suppose that ${\mathcal{L}}u\le C$ in $B_{1}$ in the viscosity sense. If $\varphi\in C^{2}(B_{\rr}(x_{0}))$ is such that
\begin{flalign*}
\varphi(x_{0})=u(x_{0}), \quad \varphi(x)\ge u(x) \ \ \text{in} \ \ B_{\rr}(x_{0})\Subset B_{1},
\end{flalign*}
for some $0<\rr<1$, then ${\mathcal{L}}u$ is defined {in the} pointwise {sense} at $x_{0}$ and ${\mathcal{L}}u(x_{0})\le C$.
\end{proposition}
\begin{proof}
For $0<\rr'\le \rr$, set
\[
\varphi_{\rr'}:=
\begin{cases}
\ \varphi  & \mbox{in \ }B_{\rr'}(x_{0})\\
\ u & \mbox{in \ }\mathds{R}^{n}\setminus B_{\rr'}(x_{0}).
\end{cases}
\]
Since $u$ is a viscosity subsolution in the sense of Definition~\ref{def_viscosity}, then 
\begin{flalign}\label{0}
{\mathcal{L}}\varphi_{\rr'}(x_{0})\le C.
\end{flalign}
The proof will now follows that of the analogous result for the $p$-Laplace equation in~\cite{Lin16}, which extends to the nonlinear case the original proof by Caffarelli and Silvestre in~\cite[Lemma~3.3]{CS09}. Clearly, we have to take into account the competition between the nonlocal kernels modulated by the coefficient~$a(\cdot,\cdot)$, and this will require some modifications. Firstly, by introducing a similar terminology to the one adopted in~\cite{Lin16}, we  define  the following quantities
\begin{eqnarray*}
\delta_{p}(\varphi_{\rr'},x,y)\!\!&:=&\!\!\frac{1}{2}\snr{\varphi_{\rr'}(x)-\varphi_{\rr'}(x+y)}^{p-2}(\varphi_{\rr'}(x)-\varphi_{\rr'}(x+y))\\*
&&\!\!+\, \frac{1}{2}\snr{\varphi_{\rr'}(x)-\varphi_{\rr'}(x-y)}^{p-2}(\varphi_{\rr'}(x)-\varphi_{\rr'}(x-y)),\\[1ex]
\delta_{q}(\varphi_{\rr'},x,y)\!\!&:=&\!\!\frac{1}{2}\snr{\varphi_{\rr'}(x)-\varphi_{\rr'}(x+y)}^{q-2}(\varphi_{\rr'}(x)-\varphi_{\rr'}(x+y))a(x,y)\\*
&&\!\!+\, \frac{1}{2}\snr{\varphi_{\rr'}(x)-\varphi_{\rr'}(x-y)}^{q-2}(\varphi_{\rr'}(x)-\varphi_{\rr'}(x-y))a(x,-y),
\end{eqnarray*}
and we consider the corresponding positive and negative part defined as follows,
\begin{eqnarray*}
\delta_{p}^{\pm}(\varphi_{\rr'},x,y) :=&\!\!\max\big\{\pm \delta_{p}(\varphi_{\rr'},x,y),\,0\big\},\\[1ex]
\delta_{q}^{\pm}(\varphi_{\rr'},x,y) :=&\!\!\max\big\{\pm \delta_{q}(\varphi_{\rr'},x,y),\,0\big\}.
\end{eqnarray*} 

Since $\varphi \in C^{2}(B_{\rr}(x_{0}))$, by keeping in mind the structural assumptions on the kernels, from~\eqref{0} we can deduce that
\begin{flalign}
 C \, \ge &\ {\mathcal{L}}\varphi_{\rr'}(x_{0})\nonumber \\*
=& \int_{\mathds{R}^{n}\setminus B_{\rr'}(0)}\snr{\varphi_{\rr'}(x_{0})-\varphi_{\rr'}(x_{0}+y)}^{p-2}(\varphi_{\rr'}(x_{0})-\varphi_{\rr'}(x_{0}+y))K_{sp}(x_0,y)dy\nonumber \\*
&+\int_{\mathds{R}^{n}\setminus B_{\rr'}(0)}a(x_{0},y)\snr{\varphi_{\rr'}(x_{0})-\varphi_{\rr'}(x_{0}+y)}^{q-2}(\varphi_{\rr'}(x_{0})-\varphi_{\rr'}(x_{0}+y))K_{tq}(x_0,y) \dy\nonumber \\*
&+ \lim_{\varepsilon\to 0^{+}}\int_{B_{\rr'}(0)\setminus B_{\varepsilon}(0)}\snr{\varphi_{\rr'}(x_{0})-\varphi_{\rr'}(x_{0}+y)}^{p-2}(\varphi_{\rr'}(x_{0})-\varphi_{\rr'}(x_{0}+y))K_{sp}(x_0,y)\dy\nonumber \\*
&+ \lim_{\varepsilon\to 0^{+}}\int_{B_{\rr'}(0)\setminus B_{\varepsilon}(0)}a(x_{0},y)\snr{\varphi_{\rr'}(x_{0})-\varphi_{\rr'}(x_{0}+y)}^{q-2}(\varphi_{\rr'}(x_{0})-\varphi_{\rr'}(x_{0}+y))K_{tq}(x_0,y)\dy\nonumber \\[1ex] 
=&\int_{\mathds{R}^{n}\setminus B_{\rr'}(0)}\big(\delta_{p}(\varphi_{\rr'},x_{0},y)K_{sp}(x_0,y)+\delta_{q}(\varphi_{\rr'},x_{0},y)K_{tq}(x_0,y)\big) \dy\nonumber \\*
&+ \lim_{\varepsilon\to 0^{+}}\int_{B_{\rr'}(0)\setminus B_{\varepsilon}(0)}\frac{1}{2}\snr{\varphi(x_{0})-\varphi(x_{0}+y)}^{p-2}(\varphi(x_{0})-\varphi(x_{0}+y))K_{sp}(x_0,y) \dy\nonumber \\*
&+\lim_{\varepsilon\to 0^{+}}\int_{B_{\rr'}(0)\setminus B_{\varepsilon}(0)}\frac{1}{2}\snr{\varphi(x_{0})-\varphi(x_{0}-y)}^{p-2}(\varphi(x_{0})-\varphi(x_{0}-y))K_{sp}(x_0,-y) \dy\nonumber \\*
&+ \lim_{\varepsilon\to 0^{+}}\int_{B_{\rr'}(0)\setminus B_{\varepsilon}(0)}\frac{a(x_{0},y)}{2}\snr{\varphi(x_{0})-\varphi(x_{0}+y)}^{q-2}(\varphi(x_{0})-\varphi(x_{0}+y))K_{tq}(x_0,y) \dy\nonumber \\*
&+ \lim_{\varepsilon\to 0^{+}}\int_{B_{\rr'}(0)\setminus B_{\varepsilon}(0)}\frac{a(x_{0},-y)}{2}\snr{\varphi(x_{0})-\varphi(x_{0}-y)}^{q-2}(\varphi(x_{0})-\varphi(x_{0}-y))K_{tq}(x_0,-y) \dy\nonumber \\[1ex]
=&\int_{\mathds{R}^{n}\setminus B_{\rr'}(0)}\big(\delta_{p}(\varphi_{\rr'},x_{0},y)K_{sp}(x_0,y)+\delta_{q}(\varphi_{\rr'},x_{0},y)K_{tq}(x_0,y)\big) \dy\nonumber \\*
&+ \int_{B_{\rr'}(0)}\big(\delta_{p}(\varphi,x_{0},y)K_{sp}(x_{0},y)+\delta_{q}(\varphi,x_{0},y)K_{tq}(x_0,y) \big) \dy \nonumber \\[1ex]
=&\int_{\mathds{R}^{n}}\big(\delta_{p}(\varphi_{\rr'},x_{0},y)K_{sp}(x_{0},y)+\delta_{q}(\varphi_{\rr'},x_{0},y)K_{tq}(x_{0},y)\big)\dy, \label{001}
\end{flalign}
where,  we first operated the change of variable $y\mapsto -y$, used \eqref{K} and then, by \eqref{pq} and the fact that $\varphi_{\rr'}\equiv \varphi$ in $B_{\rr'}$, with $\varphi \in C^{2}(B_{\rr'}(x_{0}))$, we noticed that the singular integrals appearing in the previous display actually converge.
\vspace{0.1mm}

 Moreover, given that for $0<\varsigma_{1}<\varsigma_{2}< \rr$ one has $\varphi_{\varsigma_{2}}\,\ge\, \varphi_{\varsigma_{1}}\,{\ge u}$, we can recall {the very definition of the $\delta$'s, \eqref{K} and \eqref{a}, to get} 
\begin{flalign}\label{002}
\delta_{r}(\varphi_{\varsigma_{2}},x_{0},y)\ \le\  \delta_{r}(\varphi_{\varsigma_{1}},x_{0},y)\ \le\ \delta_{r}(u,x_{0},y),
\end{flalign}
{with $r=\{p,q\}$ and $\varsigma_{1}<\varsigma_{2}< \rr$. Therefore we deduce that}
\begin{eqnarray}\label{5}
\delta_{p}^{-}(u,x_{0},y)+\delta_{q}^{-}(u,x_{0},y) \!\!& \le &\!\!\delta_{p}^{-}(\varphi_{\varsigma},x_{0},y)+\delta_{q}^{-}(\varphi_{\varsigma},x_{0},y)\nonumber\\*[1ex]
& \le&\!\! \snr{\delta_{p}(\varphi_{\rr},x_{0},y)}+\snr{\delta_{q}(\varphi_{\rr},x_{0},y)},
\end{eqnarray}
for all $0<\varsigma<\rr$. Using \eqref{pq}, the content of Lemma \ref{revL2} in the appendix and the fact that $\varphi\in C^{2}(B_{\rr}(x_{0}))$, we have that
\begin{flalign*}\snr{\delta_{p}(\varphi_{\rr},x_{0},\cdot)}K_{sp}(x_0,\cdot)+\snr{\delta_{q}(\varphi_{\rr},x_{0},\cdot)}K_{tq}(x_0,\cdot)\in L^{1}(\mathds{R}^{n}),
\end{flalign*}
and thus $\delta_{p}^{-}(u,x_{0},\cdot)K_{sp}(x_0,\cdot)+\delta_{q}^{-}(u,x_{0},\cdot)K_{tq}(x_0,\cdot)$ is integrable. 
Furthermore, we have
\begin{eqnarray}\label{003}
&& \int_{\mathds{R}^{n}}\left(\delta^{+}_{p}(\varphi_{\rr'},x_{0},y)K_{sp}(x_0,y)+\delta^{+}_{q}(\varphi_{\rr'},x_{0},y)K_{tq}(x_0,y)\right) \dy\nonumber \\*
&& \qquad\qquad \quad \ =\ \int_{\mathds{R}^{n}}\big(\delta_{p}^{-}(\varphi_{\rr'},x_{0},y)K_{sp}(x_0,y)+\delta_{p}^{-}(\varphi_{\rr'},x_{0},y)K_{sp}(x_0,y)\nonumber  \\*
&&\qquad \qquad \qquad \qquad \ \ \,+\delta_{q}(\varphi_{\rr'},x_{0},y)K_{tq}(x_0,y) + \delta_{q}^{-}(\varphi_{\rr'},x_{0},y)K_{tq}(x_0,y)\big) \dy\nonumber \\*[1ex]
&&\qquad \qquad \quad  \stackleq{001} \ \int_{\mathds{R}^{n}}\big(\delta_{p}^{-}(\varphi_{\rr'},x_{0},y)K_{sp}(x_0,y)+\delta_{q}^{-}(\varphi_{\rr'},x_{0},y)K_{tq}(x_0,y)\big)\dy+C.
\end{eqnarray}
\vspace{1mm}

Hence, for $\varsigma_{1}<\varsigma_{2}$,
\begin{eqnarray}\label{2}
&& \hspace{-1.2cm}\int_{\mathds{R}^{n}}\big(\delta^{+}_{p}(\varphi_{\varsigma_{1}},x_{0},y)K_{sp}(x_{0},y)+\delta^{+}_{q}(\varphi_{\varsigma_{1}},x_{0},y)K_{tq}(x_{0},y)\big) \dy \nonumber \\
&& \stackleq{003}\  \int_{\mathds{R}^{n}}\big(\delta^{-}_{p}(\varphi_{\varsigma_{1}},x_{0},y)K_{sp}(x_{0},y)+\delta^{-}_{q}(\varphi_{\varsigma_{1}},x_{0},y)K_{tq}(x_{0},y)\big)\dy+C\nonumber \\[1ex]
&& \stackleq{002} \ \int_{\mathds{R}^{n}}\big(\delta^{-}_{p}(\varphi_{\varsigma_{2}},x_{0},y)K_{sp}(x_{0},y)+\delta^{-}_{q}(\varphi_{\varsigma_{2}},x_{0},y)K_{tq}(x_{0},y)\big) \dy+C \ < \ \infty.
\end{eqnarray}
Since 
\begin{flalign}\label{005}
\delta^{+}_{p}(\varphi_{\rr'},x_{0},y)+\delta^{+}_{q}(\varphi_{\rr'},x_{0},y) 
\to \delta^{+}_{p}(u,x_{0},y)+\delta^{+}_{q}(u,x_{0},y) \ \text{as } {\rr'\to 0^{+}},
\end{flalign}
and {$\delta^{+}_{r}(\varphi_{\rr'},x_{0},y)=\delta_{r}(\varphi_{\rr'},x_{0},y)+\delta^{-}_{r}(\varphi_{\rr'},x_{0},y)$, from \eqref{002}, \eqref{5}, \eqref{005}, the monotone convergence theorem and the dominated one, we have}
\begin{eqnarray*}
&\dys \int_{\mathds{R}^{n}}\big(\delta^{+}_{p}(\varphi_{\rr'},x_{0},y)K_{sp}(x_{0},y)+ \delta^{+}_{q}(\varphi_{\rr'},x_{0},y)K_{tq}(x_{0},y)\big) \dy\\*
&\downarrow \\* 
&\dys \int_{\mathds{R}^{n}}\big(\delta^{+}_{p}(u,x_{0},y)K_{sp}(x_{0},y) +\ \delta^{+}_{q}(u,x_{0},y)K_{tq}(x_{0},y) \big)\dy,
\end{eqnarray*}
as $\rr'$ goes to $0^{+}$. Now, using in \eqref{2} the fact that $\varsigma\mapsto\delta^{-}_{r}(\varphi_{\varsigma},x_{0},y)$ is non decreasing, we obtain
\begin{eqnarray}\label{3}
&&\int_{\mathds{R}^{n}}\big(\delta^{+}_{p}(u,x_{0},y)K_{sp}(x_{0},y)+\delta^{+}_{q}(u,x_{0},y)K_{tq}(x_{0},y)\big) \dy\nonumber\\*
&&\qquad \quad \le \ \int_{\mathds{R}^{n}}\big(\delta^{-}_{p}(\varphi_{\rr'},x_{0},y)K_{sp}(x_{0},y)+\delta^{-}_{q}(\varphi_{\rr'},x_{0},y)K_{tq}(x_{0},y)\big) \dy+C \ < \ \infty,
\end{eqnarray}
for all $0<\rr'<\rr$. This allows us to conclude that $\delta^{+}_{p}(u,x_{0},\cdot)K_{sp}(\cdot)+\delta^{+}_{q}(u,x_{0},\cdot)K_{tq}(\cdot)$ is integrable. So that, by \eqref{5} and the {monotone} convergence theorem, we can pass to the limit as $\rr'\to 0^{+}$ in the right-hand side of \eqref{3}. We  get
\begin{eqnarray*}
&&\int_{\mathds{R}^{n}}\big(\delta^{+}_{p}(u,x_{0},y)K_{sp}(x_{0},y)+\delta^{+}_{q}(u,x_{0},y)K_{tq}(x_{0},y)\big) \dy\\*
&& \qquad \qquad \qquad \le \ \int_{\mathds{R}^{n}}\big(\delta^{-}_{p}(u,x_{0},y)K_{sp}(x_{0},y)+\delta^{-}_{q}(u,x_{0},y)K_{tq}(x_{0},y)\big) \dy+C,
\end{eqnarray*}
which means
\begin{flalign*} 
\int_{\mathds{R}^{n}}\delta_{p}(u,x_{0},y)K_{sp}(x_{0},y)+\delta_{q}(u,x_{0},y)K_{tq}(x_{0},y) \dy \ \le \ C.
\end{flalign*}
Finally, by applying $-y\mapsto y$ in the display above, we end up with
\begin{flalign*}
\int_{\mathds{R}^{n}}&\snr{u(x_{0})-u(x_{0}+y)}^{p-2}(u(x_{0})-u(x_{0}+y))K_{sp}(x_{0},y) \dy\\
&+\int_{\mathds{R}^{n}}a(x_{0},y)\snr{u(x_{0})-u(x_{0}+y)}^{q-2}(u(x_{0})-u(x_{0}+y))K_{tq}(x_{0},y) \dy
\ \le\ C;
\end{flalign*}
i.~\!e., $\mathcal{L}u(x_{0})$ exists in the pointwise sense and ${\mathcal{L}}u(x_{0})\le C$, as desired.
\end{proof}

\vspace{2mm}
\section{Fundamental regularity estimates}\label{sec_estimates}
In this section we  set the background for the proof of Theorem~\ref{teo_holder}, {whose core is a refinement  result for inhomogeneous nonlocal double phase equations given in forthcoming Lemma~\ref{lem_positivity}  which goes back to the pioneering work by~De~Giorgi. We refer the reader to   
 the important paper~\cite{Sil06} for first results in this direction for  fractional operators, which -- as mentioned in the introduction -- will be the starting point of our main proof.}
 
 \vspace{2mm}
Let us denote by $\beta$ any radial map which is~$C^{2}$-regular, vanishes outside~$B_{1}$, and it is non-increasing along rays from the origin. 
The precise expression of such function~$\beta$ is not relevant; for the sake of simplicity, one can just keep in mind the following choice (with a slight abuse of notation),
\begin{flalign*}
\beta(x)=\beta(\snr{x}):=\left((1-\snr{x}^{2})^{+}\right)^{2}.
\end{flalign*}

\vspace{1mm}

We now need to introduce a slightly modified version of the original operator in~\eqref{L}. Precisely, for a positive, absolute constant $\hat{c}$ we shall consider
\begin{eqnarray}\label{Lc}
\hat{\mathcal{L}}u(x)\!\!&:=&\!\!{P.~\!V.}\!\dys\int_{\mathds{R}^{n}}\snr{u(x)-u(x+y)}^{p-2}(u(x)-u(x+y))K_{sp}(x,y)  \dy\nonumber \\*
&&\!\!+\ {P.~\!V.}\!\dys\int_{\mathds{R}^{n}}\hat{c}a(x,y)\snr{u(x)-u(x+y)}^{q-2}(u(x)-u(x+y))K_{tq}(x,y) \dy,
\end{eqnarray}
with the corresponding problem
\begin{eqnarray}\label{problemac}
\hat{\mathcal{L}}u(x)\!\!&:=&\!\!\hat{f} \ \ \mbox{in} \ \ B_{2},
\end{eqnarray}
for $\hat{f}\in L^{\infty}(B_{2})$. Clearly, being the dilation constant~$\hat{c}$ strictly positive, all the assumptions \eqref{K}--\eqref{energiafinita} are satisfied, with the solely difference that the upper bound $M$ appearing in \eqref{a} must be replaced by $\hat{M}:=\hat{c}M$. Thus, in particular,   the results in the previous section can be plainly applied to this family of nonlocal double phase operators.

\vspace{1mm}
	
We are ready to state and prove some precise inequalities. {The statement in the proposition below could appear somewhat cumbersome, and this is so, as a natural consequence of the difficult non-uniformly growth of the involved operators as well as of their non-standard scaling and dilation properties (see also the appendix on Page~\pageref{appendix}).
However, in order to achieve the proof of the fundamental result in~Lemma~\ref{lem_positivity}, as well as for further developments in the nonlocal theory for the double phase equations, we  need to keep these  estimates as precise as follows.}
\begin{proposition}\label{prop_cumbersome}
Assume $p$, $q$, {$s$, $t$}, $a$, $K_{sp}$, $K_{tq}$ are as described in \eqref{K}--\eqref{a} {and $\hat{c}>0$ is an absolute constant}. Then, for any $\varepsilon>0$ there are $1/2\ge \kappa>0$ and $\eta>0$ such that, if $q\ge p\ge 2$ and $q>{1}/{(1-t)}$, there holds
\begin{flalign}\label{P10}
2^{q-2}&\kappa^{p-1}\int_{x+y\in B_{1}}{\snr{\beta(x)-\beta(x+y)}^{p-2}}(\beta(x)-\beta(x+y))K_{sp}(x,y) \dy\nonumber \\*
&+2^{q-2}\kappa^{q-1} \int_{x+y\in B_{1}}{\hat{c}a(x,y)\snr{\beta(x)-\beta(x+y)}^{q-1}}K_{tq}(x,y) \dy\nonumber\\*
&+2^{q-2}\int_{x+y\not \in B_{1}}\snr{\kappa \beta(x)-\kappa \beta(x+y)+2(\snr{2(x+y)}^{\eta}-1)}^{p-1}K_{sp}(x,y) \dy\nonumber \\*
&+2^{q-2}\int_{x+y\not \in B_{1}}{\hat{c}}a(x,y)\big|{\kappa \beta(x)-\kappa \beta(x+y)+2\big(\snr{2(x+y)}^{\eta}-1\big)}\big|^{q-1}K_{tq}(x,y) \dy\nonumber \\*
&+(2+\hat{c}M)2^{q-1}\int_{y \not \in B_{1/4}}\big((\snr{8y}^{\eta}-1)^{p-1}K_{sp}(x,y)+(\snr{8y}^{\eta}-1)^{q-1}K_{tq}(x,y)\big) \dy
\ \le\ \frac{\varepsilon}{\Lambda 2^{n+{s} p+q}},
\end{flalign}
for all $x \in B_{3/4}$. If $q\ge 2$, $q>{1}/{(1-t)}$ and $q\ge 2\ge p>{1}/({1-{s}})$, we have
\begin{flalign}\label{P11}
(6^{q-1}&+2^{2q-3})\kappa^{p-1} \int_{x+y\in B_{1}}\snr{\beta(x)-\beta(x+y)}^{p-1}K_{sp}(x,y) \dy\nonumber\\*
&\!\!+2^{q-2}\kappa^{q-1} \int_{x+y\in B_{1}}{\hat{c}a(x,y)\snr{\beta(x)-\beta(x+y)}^{q-1}}K_{tq}(x,y) \dy\nonumber \\*
&\!\!+(6^{q-1}+2^{2q-3})\int_{x+y\not \in B_{1}}\snr{\kappa\beta(x)-\kappa\beta(x+y)+2(\snr{2(x+y)}^{\eta}-1)}^{p-1}K_{sp}(x,y) \dy\nonumber \\*
&\!\!+2^{q-2}\int_{x+y\not \in B_{1}}{\hat{c}}a(x,y)\big|{\kappa \beta(x)-\kappa\beta(x+y)+2\big(\snr{2(x+y)}^{\eta}-1\big)\big|}^{q-1}K_{tq}(x,y) \dy\nonumber \\*
&\!\!+{2^{q-1}(2^{q-2}+\hat{c}M)}\int_{y \not \in B_{1/4}}\big((\snr{8y}^{\eta}-1)^{p-1}K_{sp}(y)+a(x,y)(\snr{8y}^{\eta}-1)^{q-1}K_{tq}(x,y)\big) \dy 
\nonumber \\[-2ex] &\ \le\  \frac{\varepsilon}{\Lambda 2^{n+{s} p+q}},
\end{flalign}
for any $x\in B_{3/4}$. Finally, if\, $2> q>1/(1-t)$ and \,$2> p>{1}/({1-{s}})$, 
\begin{flalign}\label{P12}
(3^{q-1}&+2^{q-1})\kappa^{p-1} \int_{x+y\in B_{1}}\snr{\beta(x)-\beta(x+y)}^{p-1}K_{sp}(x,y) \dy \nonumber \\*
&+(3^{q-1}+2^{q-1})\kappa^{q-1} \int_{x+y\in B_{1}}{\hat{c}}a(x,y)\snr{\beta(x)-\beta(x+y)}^{q-1}K_{tq}(x,y) \dy \nonumber \\*
&+(3^{q-1}+2^{q-1})\int_{x+y\not \in B_{1}}\big|{\kappa\beta(x)-\kappa\beta(x+y)+2\big(\snr{2(x+y)}^{\eta}-1\big)\big|}^{p-1}K_{sp}(x,y) \dy \nonumber \\*
&+(3^{q-1}+2^{q-1})\int_{x+y\not \in B_{1}}{\hat{c}}a(x,y)\big|{\kappa\beta(x)-\kappa\beta(x+y)+2\big(\snr{2(x+y)}^{\eta}-1\big)\big|}^{q-1}K_{tq}(x,y) \dy\nonumber\\
&+2^{q-1}(1+\hat{c}M)\int_{y \not \in B_{1/4}}\big((\snr{8y}^{\eta}-1)^{p-1}K_{sp}(y)+\hat{c}a(x,y)(\snr{8y}^{\eta}-1)^{q-1}K_{tq}(x,y)\big) \dy \nonumber \\[-2ex]
&\ \le\ \frac{\varepsilon}{\Lambda 2^{n+{s} p+q}},
\end{flalign}
again for all $x\in B_{3/4}$. Here $\kappa$ and $\eta$ depend on $\Lambda$, $p$, $q$, ${s}$, ${t}$, $\varepsilon$, $M$ and {$\hat{c}$}.
\end{proposition}
\begin{proof}
For the ease of notation, let us define
\begin{eqnarray*}
I_{p}\!\!&:=&\!\!{2^{q-2}\kappa^{p-1} \int_{x+y\in B_{1}}\snr{\beta(x)-\beta(x+y)}^{p-2}}(\beta(x)-\beta(x+y))K_{sp}(x,y) \dy,\\*
I_{q}\!\!&:=&\!\!{2^{q-2}\kappa^{q-1} \int_{x+y\in B_{1}}\hat{c}a(x,y)\snr{\beta(x)-\beta(x+y)}^{q-1}}K_{tq}(x,y) \dy,\\*
II_{p}\!\!&:=&\!\!2^{q-2}\int_{x+y\not \in B_{1}}\snr{\kappa \beta(x)-\kappa \beta(x+y)+2(\snr{2(x+y)}^{\eta}-1)}^{p-1}K_{sp}(x,y) \dy,\\*
II_{q}\!\!&:=&\!\!2^{q-2}\int_{x+y\not \in B_{1}}{\hat{c}}a(x,y)\big|{\kappa \beta(x)-\kappa \beta(x+y)+2\big(\snr{2(x+y)}^{\eta}-1\big)\big|}^{q-1}K_{tq}(x,y) \dy,\\*
III\!\!&:=&\!\!(2+\hat{c}M)2^{q-1}\int_{y \not \in B_{1/4}}\big((\snr{8y}^{\eta}-1)^{p-1}K_{sp}(x,y)+(\snr{8y}^{\eta}-1)^{q-1}K_{tq}(x,y)\big) \dy.
\end{eqnarray*}

\noindent
\\ \emph{Case 1: $q\ge p\ge 2$ and $q>\frac{1}{1-t}$.} The symmetry of the domain of integration and \eqref{K} allow rewriting term $I_{p}$ as
\begin{flalign}\label{rev20}
I_{p}=&2^{q-3}\kappa^{p-1}\int_{x+y\in B_{1}}\snr{\beta(x)-\beta(x+y)}^{p-2}(\beta(x)-\beta(x+y))K_{sp}(x,y) \ \dy\nonumber \\
&+2^{q-3}\kappa^{p-1}\int_{x+y\in B_{1}}\snr{\beta(x)-\beta(x-y)}^{p-2}(\beta(x)-\beta(x-y))K_{sp}(x,y) \ \dy.
\end{flalign}
Since $\beta\in C^{2}(B_{1})$, we easily see that
\begin{flalign}\label{rev21}
\nr{D\beta}_{L^\infty}\le 8 \ \ \mbox{and} \ \ \nr{D^{2}\beta}_{L^\infty}\le 16,
\end{flalign}
and, as in the proof of Lemma \ref{revL2}, we set
\begin{flalign*}
&A:=-(\beta(x)-\beta(x-y)),\\
&B:=(\beta(x)-\beta(x-y))+(\beta(x)-\beta(x+y)).
\end{flalign*}
In these terms we can manipulate \eqref{rev20} and, from Lemma \ref{revL1}, \eqref{K}, $\eqref{pq}_{1}$ and $\eqref{rev21}_{2}$ we get
\begin{flalign}\label{6}
I_{p}\le & 2^{q-3}\kappa^{p-1}\int_{x+y\in B_{1}}\left(\snr{A+B}^{p-2}(A+B)-\snr{A}^{p-2}A\right)K_{sp}(x,y) \ \dy\nonumber \\
\le &(p-1)2^{q-3}\kappa^{p-1}\int_{x+y\in B_{1}}\snr{B}(\snr{A}+\snr{B})^{p-2} \ \dy\nonumber \\
\le &c\kappa^{p-1}\int_{x+y\in B_{1}}\snr{y}^{p-n-sp} \ \dy \le c(n,p,q,s,\Lambda)\kappa^{p-1}\to_{\kappa\to 0^{+}}0.
\end{flalign}
Further, using this time \eqref{K}, $\eqref{pq}_{2}$,  \eqref{a} and $\eqref{rev21}_{1}$ we estimate
\begin{flalign}\label{7}
\snr{I_{q}}\le c\kappa^{q-1}\int_{x+y\in B_{1}}\snr{y}^{q-1-n-tq} \ \dy\le c(n,q,t,\Lambda,M,\hat{c})\kappa^{q-1}\to_{\kappa\to 0^{+}}0.
\end{flalign}
From estimates \eqref{6} and \eqref{7}, we immediately obtain that $I_{p}+I_{q}\to_{\kappa\to 0^{+}}0$.\\\\
\vspace{1mm}
Concerning the second couple of terms, we notice that, since conditions $x \in B_{3/4}$ and $x+y \not \in B_{1}$ imply $y\not \in B_{1/4}$ and therefore $\snr{x}\le 3\snr{y}$,  in view of~\eqref{K} and the dominated convergence theorem, we can estimate as follows,
\begin{eqnarray}\label{10}
II_{p}\!\! & \le &\!\! c\kappa^{p-1}\int_{y\not \in B_{1/4}}\snr{y}^{-n-{s} p} \dy+c\int_{y \not \in B_{1/4}}(\snr{8y}^{\eta}-1)^{p-1}\snr{y}^{-n-{s} p} \dy\nonumber \\*
&\le &\!\!c\kappa^{p-1}+c\int_{y \not \in B_{1/4}}(\snr{8y}^{\eta}-1)^{p-1}\snr{y}^{-n-{s} p} \dy \ \longrightarrow \ 0
\end{eqnarray}
as $\kappa\to 0^{+}$ and $\eta \to 0^{+}$ (notice also that the latter assures in particular $\eta <{sp/(p-1)}$, so that $\snr{y}^{\eta{(p-1)}-n-{s} p}$ is integrable over $\mathds{R}^{n}\setminus B_{1/4}$). Here $c=c(n,p,q,{s},\Lambda)$. Analogously, we can estimate as above the following term,\begin{eqnarray}\label{11}
II_{q}\!\! & \le &\!\! c\kappa^{q-1}\int_{y \not \in B_{1/4}}\snr{y}^{n-{t} q} \dy+c\int_{y \not \in B_{1/4}}(\snr{8y}^{\eta}-1)^{q-1}\snr{y}^{-n-{t} q} \dy\nonumber \\*
&\le &\!\!c\kappa^{q-1}+c\int_{y \not \in B_{1/4}}(\snr{8y}^{\eta}-1)^{q-1}\snr{y}^{-n-{t} q} \dy\ \longrightarrow\  0,
\end{eqnarray}
as $\kappa\to 0^{+}$ and $\eta \to 0^{+}$ (specifically, $\eta<{tq/(q-1)}$), for a non relabeled constant~$c=c(n,p,q,{t},M,\Lambda,{\hat{c}})$. Estimates \eqref{10}-\eqref{11} guarantee that $II_{p}+II_{q}\to 0$ as $\kappa,\eta \to 0^{+}$. Finally, remembering that $\eta$ is supposed to go to zero, we can assume $\eta<\min\left\{\frac{sp}{p-1},\frac{tq}{q-1}\right\}$, {thus it is easy to see that, by the dominated convergence theorem,}
\begin{equation}\label{modulo_lambda}
III\to_{\eta\to 0^{+}}0.
\end{equation}
Let us have a look to the term on the right-hand side of \eqref{P10}. By \eqref{K} and \eqref{a}, for any $A_{0}\subset B_{2}$ we have
\[
\int_{A_{0}}\big(K_{sp}(x,y)+\hat{c}a(x,y)K_{tq}(x,y) \big)\dy \ \ge \ \int_{A_{0}}\Lambda^{-1}\snr{y}^{-n-{s} p} \dy
\ \ge \ \frac{\snr{A_{0}}}{\Lambda 2^{n+{s} p}},
\]
which yields
\begin{equation}\label{13}
2^{1-q}\!\inf_{A_{0}\subset B_{2},\snr{A_{0}}>\varepsilon}\int_{A_{0}}
\big(K_{sp}(x,y)+\hat{c}a(x,y)K_{tq}(x,y)\big) \dy \ \ge \ \frac{\varepsilon}{\Lambda 2^{n+{s} p+q-1}},
\end{equation}
{so we can choose $\eta_{1}$, $\kappa_{1}$ sufficiently small in such a way that the sum $I_{p}+I_{q}+II_{p}+II_{q}+III$ can be controlled by the quantity in the right-hand side of \eqref{13}, e.~\!g., the number in~\eqref{P10}}. 

\noindent
\\ \emph{Case 2: $q\ge2\ge  p>{1}/{(1-{s})}$ and $q>\frac{1}{1-t}$.}~\,~We only need to take care of the term in the first line of \eqref{P11}, since the other terms can be estimated as in the previous case: the fact that those are multiplied by a different constant will not affect the result. Let
\begin{eqnarray*}
IV\!\!&:=&\!\!(6^{q-1}+2^{2q-3})\kappa^{p-1} \int_{x,y\in B_{1}}\snr{\beta(x)-\beta(x+y)}^{p-1}K_{sp}(x,y) \dy,
\end{eqnarray*}
We  have, by Lagrange's Theorem, $\eqref{rev21}_{1}$, $\eqref{pq}_{1}$ and \eqref{K}, that
\begin{equation*}
{IV} \ \le \ c\kappa^{p-1}\int_{B_{7/4}}\snr{y}^{p-1-n-{s} p} \dy\longrightarrow_{\kappa\to 0^{+}}0,
\end{equation*}
with $c=c(n,p,q,\Lambda,{s})$. {Again, we can fix $\eta_{2}$ and $\kappa_{2}$ small enough so that \eqref{P11} is satisfied}.

\noindent
\\ \emph{Case 3: $2\ge q>{1}/{(1-{t})}$ and $2\ge p>1/(1-s)$.}\, In this case we will only consider the term in the second line of \eqref{P12},
\begin{eqnarray*}
V\!\!&:=&\!\!(3^{q-1}+2^{q-1})\kappa^{q-1} \int_{x+y\in B_{1}}{\hat{c}}a(x,y)\snr{\beta(x)-\beta(x+y)}^{q-1}K_{tq}(x,y) \dy.
\end{eqnarray*}
By Lagrange's Theorem, $\eqref{rev21}_{1}$, \eqref{a}, $\eqref{pq}_{2}$ and \eqref{K} we obtain
\[
V\ \le\  c\int_{B_{7/4}}\snr{y}^{q-1-n-{t} q} \dy\longrightarrow_{\kappa\to 0^{+}}0,
\]
where $c=c(n,q,M,\Lambda,{t},{\hat{c}})$. {Hence, we can find $\eta_{3}$ and $\kappa_{3}$ for which \eqref{P12} holds true, and this concludes the proof.}
\end{proof}
\vspace{1mm}

\begin{remark}\label{R1}
\emph{{From now on, we will work with parameters $\dys \eta:=\min_{i=\{1,2,3\}}\eta_{i}>0$ and $\dys \kappa:=\min_{i \in \{1,2,3\}}\kappa_{i}>0$, where the $\eta_{i}$'s, $\kappa_{i}$'s are those determined in Proposition \ref{prop_cumbersome}. This choice is clearly motivated by the fact that these definitions assure that \eqref{P10}--\eqref{P12} are simultaneously matched.}}
\end{remark}

\begin{lemma}\label{lem_positivity}
Under assumptions \eqref{K}--\eqref{energiafinita}, let~$\eta$ be as the one determined in Remark~\ref{R1}, and set
\begin{flalign}\label{sigma}
\sigma:=2^{q-1}\!\int_{y \not \in B_{1/4}}\left((\snr{8y}^{\eta}-1)^{p-1}\snr{y}^{-n-{s} p}+(\snr{8y}^{\eta}-1)^{q-1}\snr{y}^{-n-{t} q}\right)\!\dy. 
\end{flalign}
If $u$ is such that~$\snr{B_{1}\cap \{x\colon u(x)\le 0\}}>\varepsilon$ and
\begin{equation}\label{ass}
\begin{cases}
\ \hat{\mathcal{L}}u\leq\sigma\quad &\mbox{in} \ B_{1},\\*[0.5ex]
\ u\le 1 \quad &\mbox{in} \ B_{1},\\*[0.5ex]
\ u(x)\le 2\snr{2x}^{\eta}-1 \quad &\mbox{in} \ \mathds{R}^{n}\setminus B_{1},\\*[0.5ex]
\end{cases}
\end{equation}
then there exists $\theta=\theta(n,p,q,s,t,M,\Lambda,{\hat{c}})>0$ such that
\begin{equation*}
u\, \le\, 1-\theta \  \text{in } B_{1/2}.
\end{equation*} 
\end{lemma}
\begin{proof}
Let $\kappa$ be as in Remark \ref{R1} and set $\theta:=\kappa(\beta(1/2)-\beta(3/4))=95\kappa/256$. By~contradiction, suppose that there is $x_{0}\in B_{1/2}$ such that $u(x_{0})>1-\theta$. Then $u(x_{0})+\kappa \beta(1/2)>1+\kappa \beta(3/4)$. Moreover, since $\beta$ is non increasing along rays, for any $y\in B_{1}\setminus B_{3/4}$ there holds 
\begin{flalign*}
u(x_{0})+\kappa\beta(x_{0})\ > \ u(x_{0})+\kappa\beta(1/2)\ >\ 1+\kappa\beta(3/4)\ \ge\  u(y)+\kappa\beta(y).
\end{flalign*}
This shows that the maximum of $u+\kappa\beta$ in $B_{1}$ is attained in $B_{3/4}$ and it is strictly larger than one. Let $\bar{x}\in B_{3/4}$ be such a maximum point. We aim to estimate~$\mathcal{L}(u+\kappa\beta)(\bar{x})$ from above and below as to reach a contradiction with \eqref{P10}--\eqref{P12}. Notice that $-\kappa\beta+(u+\kappa\beta)(\bar{x})$ touches $u$ from above at $\bar{x}$, thus, being $\beta \in C^{2}_{c}(B_{1})$, by Proposition~\ref{prop_classic},
\begin{flalign}\label{16}
\hat{\mathcal{L}}u(\bar{x})\, \le \, \sigma \ \ \mbox{in the pointwise sense}.
\end{flalign}
We begin by estimating $\hat{\mathcal{L}}(u+\kappa\beta)(\bar{x})$ from below. We then rewrite
\begin{flalign*}
\mathcal{L}(u+\kappa\beta)(\bar{x})\, =\, \lim_{\rr\to 0^{+}}I_{\rr}+I_{2},
\end{flalign*}
where
\begin{eqnarray*}
&&I_{\rr}\ :=\ \int_{\bar{x}+y\in B_{1},y \not \in B_{\rr}}\big(\snr{u(\bar{x})+\kappa\beta(\bar{x})-u(\bar{x}+y)-\kappa\beta(\bar{x}+y)}^{p-2}\\*
&&\qquad \qquad \qquad \qquad \quad \ \times(u(\bar{x})+\kappa\beta(\bar{x})-u(\bar{x}+y)-\kappa\beta(\bar{x}+y))K_{sp}(\bar{x},y)\big) \dy\\*
&&\qquad \quad \, +\, \int_{\bar{x}+y\in B_{1},y \not \in B_{\rr}}\big(\hat{c}a(\bar{x},y)\snr{u(\bar{x})+\kappa\beta(\bar{x})-u(\bar{x}+y)-\kappa\beta(\bar{x}+y)}^{q-2}\\*
&&\qquad \qquad \qquad \qquad \qquad  \ \ \times(u(\bar{x})+\kappa\beta(\bar{x})-u(\bar{x}+y)-\kappa\beta(\bar{x}+y))K_{tq}(\bar{x},y) \big)\dy,\\[1ex]
&&I_{2}\ :=\ \int_{\bar{x}+y\not \in B_{1}}\big(\snr{u(\bar{x})+\kappa\beta(\bar{x})-u(\bar{x}+y)-\kappa\beta(\bar{x}+y)}^{p-2}\\*
&&\qquad \qquad \qquad \quad \ \times(u(\bar{x})+\kappa\beta(\bar{x})-u(\bar{x}+y)-\kappa\beta(\bar{x}+y))K_{sp}(\bar{x},y)\big) \dy\\*
&&\qquad \quad \, +\, \int_{\bar{x}+y\not \in B_{1}}\big(\hat{c}a(\bar{x},y)\snr{u(\bar{x})+\kappa\beta(\bar{x})-u(\bar{x}+y)-\kappa\beta(\bar{x}+y)}^{q-2}\\*
&&\qquad \qquad \qquad \qquad \ \,  \times(u(\bar{x})+\kappa\beta(\bar{x})-u(\bar{x}+y)-\kappa\beta(\bar{x}+y))K_{tq}(\bar{x},y)\big) \dy.
\end{eqnarray*}
 Set $A_{0}:=\big\{\bar{x}+y \in B_{1}\colon u(\bar{x}+y)\le 0\big\}$; by using that $u(\bar{x})+\kappa\beta(\bar{x})>1$ is the maximum of $u+\kappa\beta$ in $B_{1}$, we see that the integrand in $I_{\rr}$ is nonnegative and we have the following estimate
\[
I_{\rr}\ \ge\ \int_{A_{0}\cap B_{\rr}^{c}}\big((1-\kappa\beta(\bar{x}+y))^{p-1}K_{sp}(\bar{x},y)+\hat{c}a(\bar{x},y)(1-\kappa\beta(\bar{x}+y))^{q-1}K_{tq}(\bar{x},y)\big) \dy,
\]
Since $\beta\le 1$ and $\kappa\le 1/2$, recalling also that $\hat{c}>0$ and, by \eqref{a}, $a(\cdot,\cdot)$ is nonnegative, we can deduce
\begin{eqnarray*}
\liminf_{\rr\to 0}I_{\rr} \!\! &\ge &\!\! 2^{1-q}\inf_{A\subset B_{2}, \snr{A}>\varepsilon}\int_{A}\big(K_{sp}(\bar{x},y)+\hat{c}a(\bar{x},y)K_{tq}(\bar{x},y)\big) \dy \\*[1ex]
& \ge & \!\!2^{1-q}\inf_{A\subset B_{2}, \snr{A}>\varepsilon}\int_{A}K_{sp}(\bar{x},y) \dy
\ \ge\ \frac{\varepsilon}{\Lambda 2^{n+{s} p+q-1}}.
\end{eqnarray*}
Let us take care of $I_{2}$. Exploiting that $u(\bar{x})+\kappa\beta(\bar{x})>1$, $\eqref{ass}_{3}$ and that $\beta\equiv 0$ in~$\mathds{R}^{n}\setminus B_{1}$, we have
\begin{eqnarray*}
I_{2}\!\!&\ge &\!\!\int_{\bar{x}+y\not \in B_{1}}2^{p-1}\big|{1-\snr{2(\bar{x}+y)}^{\eta}}\big|^{p-2}\big(1-\snr{2(\bar{x}+y)}^{\eta}\big)K_{sp}(\bar{x},y) \dy\\*
&&\!\! +\ \int_{\bar{x}+y\not \in B_{1}}2^{q-1}\hat{c}a(\bar{x},y)\big|{1-\snr{2(\bar{x}+y)}^{\eta}}\big|^{q-2}\big(1-\snr{2(\bar{x}+y)}^{\eta}\big)K_{tq}(\bar{x},y) \dy\\[1ex]
&\ge &\!\!\int_{y \not \in B_{1/4}}2^{p-1}\left |\  1-\left| \ 2\left(\frac{3}{4}+\snr{y}\right) \ \right |^{\eta} \ \right |^{p-2}\left(1-\left | \ 2\left(\frac{3}{4}+\snr{y}\right) \  \right |^{\eta}\right)K_{sp}(\bar{x},y) \dy\\*
&&\!\! +\ \int_{y \not \in B_{1/4}}2^{q-1}\hat{c}a(\bar{x},y)\left |\  1-\left| \ 2\left(\frac{3}{4}+\snr{y}\right) \ \right |^{\eta} \ \right |^{q-2}\left(1-\left | \ 2\left(\frac{3}{4}+\snr{y}\right) \  \right |^{\eta}\right)K_{tq}(\bar{x},y) \dy\\[1ex]
&\ge &\!\!-2^{q-1}\int_{y \not \in B_{1/4}}\big((\snr{8y}^{\eta}-1)^{p-1}K_{sp}(\bar{x},y)+\hat{c}a(\bar{x},y)(\snr{8y}^{\eta}-1)^{q-1}K_{tq}(\bar{x},y)\big) \dy.
\end{eqnarray*}
Adding together the content of the  two displays above, we obtain
\begin{eqnarray}\label{14}
\mathcal{L}(u+\kappa\beta)(\bar{x}) 
\!\!& \ge &\!\! \frac{\varepsilon}{\Lambda 2^{n+{s} p+q-1}} \nonumber \\*
&& \!\! -\ 2^{q-1}\int_{y \not \in B_{1/4}}\big((\snr{8y}^{\eta}-1)^{p-1}K_{sp}(\bar{x},y)+\hat{c}a(\bar{x},y)(\snr{8y}^{\eta}-1)^{q-1}K_{tq}(\bar{x},y)\big) \dy.
\end{eqnarray}

Now, we want to estimate $\mathcal{L}(u+\kappa\beta)(\bar{x})$ from above, and for this we need to distinguish the following three cases: $q\ge p\ge 2$,\, $q\ge 2\ge p$, and $2\ge q\ge p$.\\\\
\emph{Case 1: $q\ge p\ge 2$ and $q>{1}/{(1-t)}$.} We split the integral into two parts:
\[
\mathcal{L}(u+\kappa\beta)(\bar{x})\,=\,I_{1}+I_{2},
\]
where
\begin{eqnarray*}
&&I_{1}\ :=\ \int_{\bar{x}+y\in B_{1}}
\big(\snr{u(\bar{x})+\kappa\beta(\bar{x})-u(\bar{x}+y)-\kappa\beta(\bar{x}+y)}^{p-2}\\*
&&\qquad\qquad \qquad \quad \times(u(\bar{x})+\kappa\beta(\bar{x})-u(\bar{x}+y)-\kappa\beta(\bar{x}+y))K_{sp}(\bar{x},y)\big) \dy\\*
&&\qquad \quad+\ \int_{\bar{x}+y \in B_{1}}
\big(\hat{c}a(\bar{x},y)\snr{u(\bar{x})+\kappa\beta(\bar{x})-u(\bar{x}+y)-\kappa\beta(\bar{x}+y)}^{q-2}\\*
&&\qquad \quad\qquad\qquad \quad \ \times(u(\bar{x})+\kappa\beta(\bar{x})-u(\bar{x}+y)-\kappa\beta(\bar{x}+y))K_{tq}(\bar{x},y)\big) \dy,\\
&&I_{2}\ :=\ \int_{\bar{x}+y\not \in B_{1}}
\big(\snr{u(\bar{x})+\kappa\beta(\bar{x})-u(\bar{x}+y)-\kappa\beta(\bar{x}+y)}^{p-2}\\*
&&\qquad\qquad \qquad \quad \times(u(\bar{x})+\kappa\beta(\bar{x})-u(\bar{x}+y)-\kappa\beta(\bar{x}+y))K_{sp}(\bar{x},y)\big) \dy\\
&&\qquad \quad +\ \int_{\bar{x}+y\not \in B_{1}}\big(\hat{c}a(\bar{x},y)\snr{u(\bar{x})+\kappa\beta(\bar{x})-u(\bar{x}+y)-\kappa\beta(\bar{x}+y)}^{q-2}\\*
&&\qquad \quad\qquad\qquad \quad \ \times(u(\bar{x})+\kappa\beta(\bar{x})-u(\bar{x}+y)-\kappa\beta(\bar{x}+y))K_{tq}(\bar{x},y) \big)\dy,
\end{eqnarray*}

We start from $I_{1}$ by noticing that when $\bar{x}+y\in B_{1}$, \begin{flalign}\label{25}
u(\bar{x})+\kappa\beta(\bar{x})-u(\bar{x}+y)-\kappa\beta(\bar{x}+y)\, \ge\, 0,
\end{flalign}
given that $u+\kappa \beta$ attains its maximum in $B_{1}$ in $\bar{x}$. So we can apply Lemma \ref{lem_superlinear} to get
\begin{eqnarray*}
&&\left|u(\bar{x})+\kappa\beta(\bar{x})-u(\bar{x}+y)-\kappa\beta(\bar{x}+y)\right |^{r-2}(u(\bar{x})+\kappa\beta(\bar{x})-u(\bar{x}+y)-\kappa\beta(\bar{x}+y))\\[1ex]
&&\quad\qquad \qquad  \le \ 2^{q-2}\big(\snr{u(\bar{x})-u(\bar{x}+y)}^{r-2}(u(\bar{x})-u(\bar{x}+y))+\snr{\kappa\beta(\bar{x})-\kappa\beta(\bar{x}+y)}^{r-2}\\
&& \quad\qquad \qquad \qquad\quad \ \, \times(\kappa\beta(\bar{x})-\kappa\beta(\bar{x}+y))\big),
\end{eqnarray*}
for $r \in \{p,q\}$. Hence, recalling also \eqref{a} and the fact that $\hat{c}>0$,
\begin{eqnarray*}
I_{1}&\le& 2^{q-2} \int_{\bar{x}+y\in B_{1}}\snr{u(\bar{x})-u(\bar{x}+y)}^{p-2}(u(\bar{x})-u(\bar{x}+y))K_{sp}(\bar{x},y) \dy\\
&&+\ 2^{q-2}\kappa^{p-1} \int_{\bar{x}+y\in B_{1}}\snr{\beta(\bar{x})-\beta(\bar{x}+y)}^{p-2}(\beta(\bar{x})-\beta(\bar{x}+y))K_{sp}(\bar{x},y) \dy\\
&&+\ 2^{q-2} \int_{\bar{x}+y\in B_{1}}\hat{c}a(\bar{x},y)\snr{u(\bar{x})-u(\bar{x}+y)}^{q-2}(u(\bar{x})-u(\bar{x}+y))K_{tq}(\bar{x},y) \dy\\
&&+\ 2^{q-2}\kappa^{q-1} \int_{\bar{x}+y\in B_{1}}\hat{c}a(\bar{x},y)\snr{\beta(\bar{x})-\beta(\bar{x}+y)}^{q-1}K_{tq}(\bar{x},y) \dy.
\end{eqnarray*}

Now let us consider $I_{2}$. Here it is not possible to apply directly Lemma \ref{lem_superlinear}, but from the assumptions we still have
\begin{flalign*}
u(\bar{x})+\kappa\beta(\bar{x})\,>\,1 \quad \mbox{and}\quad u(\bar{x}+y)+\kappa\beta(\bar{x}+y)\,\le\, 2\snr{2(\bar{x}+y)}^{\eta}-1,
\end{flalign*}
which means that
\begin{flalign*}
u(\bar{x})-u(\bar{x}+y)+\kappa\beta(\bar{x})-\kappa\beta(\bar{x}+y)\, >\, 2(1-\snr{2(\bar{x}+y)}^{\eta}).
\end{flalign*}
Adding $2(\snr{2(\bar{x}+y)}^{\eta}-1)>0$ to both sides of the preceding inequality we increase the integrand and we make it nonnegative:
\begin{flalign}\label{26}
u(\bar{x})-u(\bar{x}+y)+\kappa\beta(\bar{x})-\kappa\beta(\bar{x}+y)+2(\snr{2(\bar{x}+y)}^{\eta}-1)\, \ge \, 0
\end{flalign}
so, again, we can use Lemma \ref{lem_superlinear}. We obtain
\begin{eqnarray*}
I_{2}&\le &\int_{\bar{x}+y\not \in B_{1}}\big(\snr{u(\bar{x})-u(\bar{x}+y)+\kappa\beta(\bar{x})-\kappa\beta(\bar{x}+y)+2(\snr{2(\bar{x}+y)}^{\eta}-1)}^{p-2}\\*
&&\qquad \qquad \ \times (u(\bar{x})-u(\bar{x}+y)+\kappa\beta(\bar{x})-\kappa\beta(\bar{x}+y)+2(\snr{2(\bar{x}+y)}^{\eta}-1))K_{sp}(\bar{x},y) \big)\dy \\*
&&+\ \int_{\bar{x}+y\not \in B_{1}}\big(\hat{c}a(\bar{x},y)\snr{u(\bar{x})-u(\bar{x}+y)+\kappa\beta(\bar{x})-\kappa\beta(\bar{x}+y)+2(\snr{2(\bar{x}+y)}^{\eta}-1)}^{q-2}\\*
&&\qquad \qquad \quad\ \times \big(u(\bar{x})-u(\bar{x}+y)+\kappa\beta(\bar{x})-\kappa\beta(\bar{x}+y)+2(\snr{2(\bar{x}+y)}^{\eta}-1)\big)K_{tq}(\bar{x},y)\big) \dy\\[1ex]
&\le & 2^{q-2}\int_{\bar{x}+y\not \in B_{1}}\snr{u(\bar{x})-u(\bar{x}+y)}^{p-2}(u(\bar{x})-u(\bar{x}-y))K_{sp}(\bar{x},y) \dy\\*
&&+\ 2^{q-2}\int_{\bar{x}+y\not \in B_{1}}\hat{c}a(\bar{x},y)\snr{u(\bar{x})-u(\bar{x}+y)}^{q-2}(u(\bar{x})-u(\bar{x}+y))K_{tq}(\bar{x},y) \dy\\*
&&+\ 2^{q-2}\int_{\bar{x}+y\not \in B_{1}}
\big(\snr{\kappa \beta(\bar{x})-\kappa\beta(\bar{x}+y)+2(\snr{2(\bar{x}+y)}^{\eta}-1)}^{p-2}\\*
&&\qquad \qquad \qquad \quad \times (\kappa \beta(\bar{x})-\kappa\beta(\bar{x}+y)+2(\snr{2(\bar{x}+y)}^{\eta}-1))K_{sp}(\bar{x},y)\big) \dy\\*
&&+\ 2^{q-2}\int_{\bar{x}+y\not \in B_{1}}
\big(\hat{c}a(\bar{x},y)\snr{\kappa \beta(\bar{x})-\kappa\beta(\bar{x}+y)+2(\snr{2(\bar{x}+y)}^{\eta}-1)}^{q-2}\\*
&&\qquad \qquad \qquad\quad\, \times (\kappa \beta(\bar{x})-\kappa\beta(\bar{x}+y)+2(\snr{2(\bar{x}+y)}^{\eta}-1))K_{tq}(\bar{x},y)\big) \dy.
\end{eqnarray*}
Adding the estimates for terms $I_{1}$ and $I_{2}$ and recalling \eqref{16}, we end up with
\begin{flalign}\label{15}
&\mathcal{L}(u+\kappa\beta)(\bar{x}) \nonumber \\*
&\le \ \mathcal{L}u(\bar{x})+2^{q-2}\kappa^{p-1} \int_{\bar{x}+y\in B_{1}}\snr{\beta(\bar{x})-\beta(\bar{x}+y)}^{p-2}(\beta(\bar{x})-\beta(\bar{x}+y))K_{sp}(\bar{x},y) \dy\nonumber \\*
&\quad \ +\, 2^{q-2}\kappa^{q-1} \int_{\bar{x}+y\in B_{1}}\hat{c}a(\bar{x},y)\snr{\beta(\bar{x})-\beta(\bar{x}+y)}^{q-2}(\beta(\bar{x})-\beta(\bar{x}+y))K_{tq}(\bar{x},y) \dy\nonumber \\*
&\quad \ +\, 2^{q-2}\int_{\bar{x}+y\not \in B_{1}}\snr{\kappa \beta(\bar{x})-\kappa\beta(\bar{x}+y)+2(\snr{2(\bar{x}+y)}^{\eta}-1)}^{p-2}\nonumber \\*
&&\hspace{3cm} \times (\kappa \beta(\bar{x})-\kappa\beta(\bar{x}+y)+2(\snr{2(\bar{x}+y)}^{\eta}-1))K_{sp}(\bar{x},y) \dy\nonumber \\*
&\quad\ +\, 2^{q-2}\int_{\bar{x}+y\not \in B_{1}}\hat{c}a(\bar{x},y)\snr{\kappa \beta(\bar{x})-\kappa\beta(\bar{x}+y)+2(\snr{2(\bar{x}+y)}^{\eta}-1)}^{q-2}\nonumber \\*
&\hspace{3cm} \times(\kappa \beta(\bar{x})-\kappa\beta(\bar{x}+y)+2(\snr{2(\bar{x}+y)}^{\eta}-1))K_{tq}(\bar{x},y) \dy\nonumber \\[1ex]
& \le \ \sigma + 2^{q-2}\kappa^{p-1} \int_{\bar{x}+y\in B_{1}}\snr{\beta(\bar{x})-\beta(\bar{x}+y)}^{p-1}K_{sp}(\bar{x},y) \dy\nonumber \\*
&\quad \ +\ 2^{q-2}\kappa^{q-1} \int_{\bar{x}+y\in B_{1}}\hat{c}a(\bar{x},y)\snr{\beta(\bar{x})-\beta(\bar{x}+y)}^{q-1}K_{tq}(\bar{x},y) \dy\nonumber \\*
&\quad \ +\ 2^{q-2}\int_{\bar{x}+y\not \in B_{1}}\big|{\kappa \beta(\bar{x})-\kappa\beta(\bar{x}+y)+2\big(\snr{2(\bar{x}+y)}^{\eta}-1\big)}\big|^{p-1}K_{sp}(\bar{x},y) \dy\nonumber \\*
&\quad\  +\ 2^{q-2}\int_{\bar{x}+y\not \in B_{1}}\hat{c}a(\bar{x},y)\big|{\kappa \beta(\bar{x})-\kappa\beta(\bar{x}+y)+2\big(\snr{2(\bar{x}+y)}^{\eta}-1\big)}\big|^{q-1}K_{tq}(\bar{x},y) \dy,
\end{flalign}
where we also used~\eqref{ass}$_1$.
Coupling \eqref{14} and \eqref{15}, together with the definition of~$\sigma$, and also recalling~\eqref{modulo_lambda}, 
we conclude that
\begin{eqnarray*}
\frac{\varepsilon}{\Lambda 2^{n+{s} p+q-1}}\!\!&\le &\!\!(2+\hat{c}M)2^{q-1}\int_{y \not \in B_{1/4}}\big((\snr{8y}^{\eta}-1)^{p-1}K_{sp}(\bar{x},y)+(\snr{8y}^{\eta}-1)^{q-1}K_{tq}(\bar{x},y)\big) \dy\\*
&&\!\! +\ 2^{q-2}\kappa^{p-1} \int_{\bar{x}+y\in B_{1}}\snr{\beta(\bar{x})-\beta(\bar{x}+y)}^{p-2}(\beta(x)-\beta(x+y))K_{sp}(\bar{x},y) \dy\nonumber \\*
&&\!\! +\ 2^{q-2}\kappa^{q-1} \int_{\bar{x}+y\in B_{1}}\hat{c}a(\bar{x},y)\snr{\beta(\bar{x})-\beta(\bar{x}+y)}^{q-1}K_{tq}(\bar{x},y) \dy\nonumber \\*
&&\!\! +\ 2^{q-2}\int_{\bar{x}+y\not \in B_{1}}\big|{\kappa \beta(\bar{x})-\kappa\beta(\bar{x}+y)+2\big(\snr{2(\bar{x}+y)}^{\eta}-1\big)}\big|^{p-1}K_{sp}(\bar{x},y) \dy\nonumber \\*
&&\!\! +\ 2^{q-2}\int_{\bar{x}+y\not \in B_{1}}\hat{c}a(\bar{x},y)\big|{\kappa \beta(\bar{x})-\kappa\beta(\bar{x}+y)+2\big(\snr{2(\bar{x}+y)}^{\eta}-1\big)}\big|^{q-1}K_{tq}(\bar{x},y) \dy\nonumber\\*
&\dys \stackleq{P10}&\!\!\frac{\varepsilon}{\Lambda 2^{n+{s} p+q}},
\end{eqnarray*}
which is clearly a contradiction.\\\\
\emph{Case 2: $q\ge 2\ge p>{1}/{(1-{s})}$ and $q>{1}/{(1-t)}$.}~We look again at terms $I_{1}$, $I_{2}$ and we split them as
\begin{flalign*}
I_{1}=I_{1}^{p}+I_{1}^{q}\quad \mbox{and}\quad I_{2}=I_{2}^{p}+I_{2}^{q},
\end{flalign*}
where 
\begin{eqnarray*}
I_{1}^{p}\!\!&:=&\!\! \int_{\bar{x}+y\in B_{1}}\snr{u(\bar{x})+\kappa\beta(\bar{x})-u(\bar{x}+y)-\kappa\beta(\bar{x}+y)}^{p-2}\\*
&&\qquad \quad  \ \times(u(\bar{x})+\kappa\beta(\bar{x})-u(\bar{x}+y)-\kappa\beta(\bar{x}+y))K_{sp}(\bar{x},y) \dy,\\
I_{1}^{q}\!\!&:=&\!\! \int_{\bar{x}+y \in B_{1}}\hat{c}a(\bar{x},y)\snr{u(\bar{x})+\kappa\beta(\bar{x})-u(\bar{x}+y)-\kappa\beta(\bar{x}+y)}^{q-2}\\*
&&\qquad\quad \, \, \times(u(\bar{x})+\kappa\beta(\bar{x})-u(\bar{x}+y)-\kappa\beta(\bar{x}+y))K_{tq}(\bar{x},y) \dy,\\
I_{2}^{p}\!\!&:=&\!\!\int_{\bar{x}+y\not \in B_{1}}\snr{u(\bar{x})+\kappa\beta(\bar{x})-u(\bar{x}+y)-\kappa\beta(\bar{x}+y)}^{p-2}\\*
&&\qquad\quad\, \times(u(\bar{x})+\kappa\beta(\bar{x})-u(\bar{x}+y)-\kappa\beta(\bar{x}+y))K_{sp}(\bar{x},y) \dy,\\
I_{2}^{q}\!\!&:=&\!\!\int_{\bar{x}+y\not \in B_{1}}\hat{c}a(\bar{x},y)\snr{u(\bar{x})+\kappa\beta(\bar{x})-u(\bar{x}+y)-\kappa\beta(\bar{x}+y)}^{q-2}\\*
&&\qquad \quad \, \,(u(\bar{x})+\kappa\beta(\bar{x})-u(\bar{x}+y)-\kappa\beta(\bar{x}+y))K_{tq}(\bar{x},y) \dy.
\end{eqnarray*}
Being $q\ge 2$, the estimates for terms $I_{1}^{q}$, $I_{2}^{q}$ are the same as those in the previous case, therefore we have
\begin{eqnarray}\label{22}
I_{1}^{q}\!\!&\le &\!\!2^{q-2} \int_{\bar{x}+y\in B_{1}}\hat{c}a(\bar{x},y)\snr{u(\bar{x})-u(\bar{x}+y)}^{q-2}(u(\bar{x})-u(\bar{x}+y))K_{tq}(\bar{x},y) \dy\nonumber \\*
&&\!\!+\ 2^{q-2}\kappa^{q-1} \int_{\bar{x}+y\in B_{1}}\hat{c}a(\bar{x},y)\snr{\beta(\bar{x})-\beta(\bar{x}+y)}^{q-1}K_{tq}(\bar{x},y) \dy
\end{eqnarray}
and 
\begin{eqnarray}\label{23}
I_{2}^{q}\!\!&\le &\!\!2^{q-2}\int_{\bar{x}+y\not \in B_{1}}\hat{c}a(\bar{x},y)\snr{u(\bar{x})-u(\bar{x}+y)}^{q-2}(u(\bar{x})-u(\bar{x}+y))K_{tq}(\bar{x},y) \dy\nonumber\\*
&&\!\!+\ 2^{q-2}\int_{\bar{x}+y\not \in B_{1}}\hat{c}a(\bar{x},y)\snr{\kappa \beta(\bar{x})-\kappa\beta(\bar{x}+y)+2(\snr{2(\bar{x}+y)}^{\eta}-1)}^{q-2}\nonumber \\*
&&\qquad \qquad \qquad\ \times \big(\kappa \beta(\bar{x})-\kappa\beta(\bar{x}+y)+2(\snr{2(\bar{x}+y)}^{\eta}-1)\big)K_{tq}(\bar{x},y) \dy.
\end{eqnarray}
Since $p\le 2$ and $q\ge2$, we can use \eqref{25}, \eqref{26} and Lemma \ref{lem_singular} to get
\begin{eqnarray*}
&& \big|u(\bar{x})+\kappa\beta(\bar{x})-u(\bar{x}+y)-\kappa\beta(\bar{x}+y)\big|^{p-2}\big(u(\bar{x})+\kappa\beta(\bar{x})-u(\bar{x}+y)-\kappa\beta(\bar{x}+y)\big)\nonumber \\*
&& \qquad\qquad\qquad \qquad \le\ 2^{q-2}\snr{u(\bar{x})-u(\bar{x}+y)}^{p-2}(u(\bar{x})-u(\bar{x}+y)) \\*
&&\qquad\qquad \qquad \qquad \quad \ +\ (6^{q-1}+2^{2q-3})\kappa^{p-1}\snr{\beta(\bar{x})-\beta(\bar{x}+y)}^{p-1} \nonumber
\end{eqnarray*}
and
\begin{eqnarray}\label{28}
&& \big|u(\bar{x})+\kappa\beta(\bar{x})-u(\bar{x}+y)-\kappa\beta(\bar{x}+y)+2\big(\snr{2(\bar{x}+y)}^{\eta}-1\big)\big|^{p-2}\nonumber \\*
&&\qquad \qquad\qquad\qquad\qquad \times\big(u(\bar{x})+\kappa\beta(\bar{x})-u(\bar{x}+y)-\kappa\beta(\bar{x}+y)+2\big(\snr{2(\bar{x}+y)}^{\eta}-1\big)\big)\nonumber\\*
&&\qquad\qquad\qquad \le\ 2^{q-2}\snr{u(\bar{x})-u(\bar{x}+y)}^{p-2}(u(\bar{x})-u(\bar{x}+y))\\*
&&\qquad\qquad\qquad \quad \ +\ (6^{q-1}+2^{2q-3})\big|{\kappa\beta(\bar{x})-\kappa\beta(\bar{x}+y)+2\big(\snr{2(\bar{x}+y)}^{\eta}-1\big)}\big|^{p-1}.\nonumber
\end{eqnarray}
Finally, we can conclude that
\begin{flalign}
I_{1}\stackrel{\eqref{22},\eqref{26}}{\le}&2^{q-2} \int_{\bar{x}+y\in B_{1}}\snr{u(\bar{x})-u(\bar{x}+y)}^{p-2}(u(\bar{x})-u(\bar{x}+y))K_{sp}(\bar{x},y) \dy\nonumber\\*
&+2^{q-2} \int_{\bar{x}+y\in B_{1}}\hat{c}a(\bar{x},y)\snr{u(\bar{x})-u(\bar{x}+y)}^{q-2}(u(\bar{x})-u(\bar{x}+y))K_{tq}(\bar{x},y) \dy\nonumber\\*
&+(6^{q-1}+2^{2q-3})\kappa^{p-1} \int_{\bar{x}+y\in B_{1}}\snr{\beta(\bar{x})-\beta(\bar{x}+y)}^{p-1}K_{sp}(\bar{x},y) \dy\nonumber\\*[-1.5ex]
&+2^{q-2}\kappa^{q-1} \int_{\bar{x}+y\in B_{1}}\hat{c}a(\bar{x},y)\snr{\beta(\bar{x})-\beta(\bar{x}+y)}^{q-2}(\beta(\bar{x})-\beta(\bar{x}+y))K_{tq}(\bar{x},y) \dy\label{29}
\end{flalign}
and
\begin{eqnarray}
I_{2}&\stackrel{\eqref{23},\eqref{28}}{\le}&\!\! 2^{q-2} \int_{\bar{x}+y\not \in B_{1}}\snr{u(\bar{x})-u(\bar{x}+y)}^{p-2}(u(\bar{x})-u(\bar{x}+y))K_{sp}(\bar{x},y) \dy\nonumber\\*
&&+\ 2^{q-2}\int_{\bar{x}+y\not \in B_{1}}\hat{c}a(\bar{x},y)\snr{u(\bar{x})-u(\bar{x}+y)}^{q-2}(u(\bar{x})-u(\bar{x}+y))K_{tq}(\bar{x},y) \dy\nonumber \\*
&&+\ (6^{q-1}+2^{2q-3})\int_{\bar{x}+y\not \in B_{1}}\big|{\kappa\beta(\bar{x})-\kappa\beta(\bar{x}+y)+2\big(\snr{2(\bar{x}+y)}^{\eta}-1\big)}\big|^{p-1}K_{sp}(\bar{x},y) \dy\nonumber \\*
&&+2^{q-2}\int_{\bar{x}+y\not \in B_{1}}\hat{c}a(\bar{x},y)\big|{\kappa \beta(\bar{x})-\kappa\beta(\bar{x}+y)+2\big(\snr{2(\bar{x}+y)}^{\eta}-1\big)}\big|^{q-2}\nonumber \\*
&&\qquad\qquad\qquad\ \ \times \big(\kappa \beta(\bar{x})-\kappa\beta(\bar{x}+y)+2\big(\snr{2(\bar{x}+y)}^{\eta}-1\big)\big)K_{tq}(\bar{x},y) \dy,\label{30}
\end{eqnarray}
so that
\begin{eqnarray}
&& \hspace{-2mm}\mathcal{L}(u+\kappa\beta)(\bar{x})  \nonumber\\*
&&\quad  \stackrel{\eqref{29},\eqref{30}}{\le}\!\!\!\!\!  2^{q-2}\mathcal{L}u(\bar{x})+(6^{q-1}+2^{2q-3})\kappa^{p-1} \int_{\bar{x}+y\in B_{1}}\snr{\beta(\bar{x})-\beta(\bar{x}+y)}^{p-1}K_{sp}(\bar{x},y) \dy\nonumber\\*
&&\qquad\qquad \ \  +\ 2^{q-2}\kappa^{q-1} \int_{\bar{x}+y\in B_{1}}\hat{c}a(\bar{x},y)\snr{\beta(\bar{x})-\beta(\bar{x}+y)}^{q-2}(\beta(\bar{x})-\beta(\bar{x}+y))K_{tq}(\bar{x},y) \dy\nonumber \\*
&&\qquad\qquad \ \  +\ (6^{q-1}+2^{2q-3})\int_{\bar{x}+y\not \in B_{1}}\big|{\kappa\beta(\bar{x})-\kappa\beta(\bar{x}+y)+2\big(\snr{2(\bar{x}+y)}^{\eta}-1\big)}\big|^{p-1}K_{sp}(\bar{x},y) \dy\nonumber \\*
&&\qquad\qquad \ \  +\ 2^{q-2}\int_{\bar{x}+y\not \in B_{1}}\big(\hat{c}a(\bar{x},y)\big|{\kappa \beta(\bar{x})-\kappa\beta(\bar{x}+y)+2\big(\snr{2(\bar{x}+y)}^{\eta}-1\big)}\big|^{q-2}\nonumber \\*
&&\hspace{4.5cm}\times \big(\kappa \beta(\bar{x})-\kappa\beta(\bar{x}+y)+2\big(\snr{2(\bar{x}+y)}^{\eta}-1\big)\big)K_{tq}(\bar{x},y) \big)\dy\nonumber \\[1ex]
&& \qquad \stackrel{\eqref{16}}{\le} \ \,  2^{q-2}\sigma +
 (6^{q-1}+2^{2q-3})\kappa^{p-1} \int_{\bar{x}+y\in B_{1}}\snr{\beta(\bar{x})-\beta(\bar{x}+y)}^{p-1}K_{sp}(\bar{x},y) \dy\nonumber\\*
&&\qquad\qquad \ \  +\ 2^{q-2}\kappa^{q-1} \int_{\bar{x}+y\in B_{1}}\hat{c}a(\bar{x},y)\snr{\beta(\bar{x})-\beta(\bar{x}+y)}^{q-1}K_{tq}(\bar{x},y) \dy\nonumber \\*
&&\qquad\qquad \ \  +\ (6^{q-1}+2^{2q-3})\int_{\bar{x}+y\not \in B_{1}}\big|{\kappa\beta(\bar{x})-\kappa\beta(\bar{x}+y)+2\big(\snr{2(\bar{x}+y)}^{\eta}-1\big)}\big|^{p-1}K_{sp}(\bar{x},y) \dy\nonumber \\*[-1.5ex]
&& \label{31}\\*[-1.5ex]
&&\qquad\qquad \ \  + \ 2^{q-2}\int_{\bar{x}+y\not \in B_{1}}\hat{c}a(\bar{x},y)\big|{\kappa \beta(\bar{x})-\kappa\beta(\bar{x}+y)+2\big(\snr{2(\bar{x}+y)}^{\eta}-1\big)}\big|^{q-1}K_{tq}(\bar{x},y) \dy.\nonumber
\end{eqnarray}
Recalling the definition of $\sigma$ and combining \eqref{31}, \eqref{14}, \eqref{P11}, we obtain
\begin{eqnarray*}
\frac{\varepsilon}{\Lambda 2^{n+sp+q-1}}
\!\! &\le & \!\! 2^{q-1}(2^{q-2}+\hat{c}M)\int_{y \not \in B_{1/4}}(\snr{8y}^{\eta}-1)^{p-1}K_{sp}(\bar{x},y)+(\snr{8y}^{\eta}-1)^{q-1}K_{tq}(\bar{x},y) \ \dy\\*
&&\!\!+\, (6^{q-1}+2^{2q-3})\kappa^{p-1} \int_{\bar{x}+y\in B_{1}}\snr{\beta(\bar{x})-\beta(\bar{x}+y)}^{p-1}K_{sp}(\bar{x},y) \dy\\*
&&\!\!+\,  2^{q-2}\kappa^{q-1} \int_{\bar{x}+y\in B_{1}}\hat{c}a(\bar{x},y)\snr{\beta(\bar{x})-\beta(\bar{x}+y)}^{q-1}K_{tq}(\bar{x},y) \dy\\*
&&\!\!+\, (6^{q-1}+2^{2q-3})\int_{\bar{x}+y\not \in B_{1}}\snr{\kappa\beta(\bar{x})-\kappa\beta(\bar{x}+y)+2(\snr{2(\bar{x}+y)}^{\eta}-1)}^{p-1}K_{sp}(\bar{x},y) \dy\nonumber\\*
&&\!\!+\, 2^{q-2}\int_{\bar{x}+y\not \in B_{1}}\hat{c}a(\bar{x},y)\snr{\kappa \beta(\bar{x})-\kappa\beta(\bar{x}+y)+2(\snr{2(\bar{x}+y)}^{\eta}-1)}^{q-1}K_{tq}(\bar{x},y) \dy\nonumber \\
&\le &\!\! \frac{\varepsilon}{\Lambda 2^{n+sp+q}},
\end{eqnarray*}
which again is a contradiction.\\\\
\emph{Case 3: $2\ge q>1/(1-t)$ and $ 2\ge p>{1}/{(1-{s})}$.} In this case, we first use Lemma \ref{lem_singular} to get, for $r=\{p,q\}$,
\begin{eqnarray*}
&& \left |u(\bar{x})+\kappa\beta(\bar{x})-u(\bar{x}+y)-\kappa\beta(\bar{x}+y)\right|^{r-2}\big(u(\bar{x})+\kappa\beta(\bar{x})-u(\bar{x}+y)-\kappa\beta(\bar{x}+y)\big)\nonumber \\*
&&\qquad \le \ \snr{u(\bar{x})-u(\bar{x}+y)}^{r-2}(u(\bar{x})-u(\bar{x}+y))+(3^{q-1}+2^{q-1})\kappa^{r-1}\snr{\beta(\bar{x})-\beta(\bar{x}+y)}^{r-1}
\end{eqnarray*}
and so,
\begin{eqnarray*}
&& \left |u(\bar{x})+\kappa\beta(\bar{x})-u(\bar{x}+y)-\kappa\beta(\bar{x}+y)+2(\snr{2(x+y)}^{\eta}-1)\right|^{r-2}\nonumber \\*
&&\qquad\qquad\qquad\qquad\qquad\times(u(\bar{x})+\kappa\beta(\bar{x})-u(\bar{x}+y)-\kappa\beta(\bar{x}+y)+2(\snr{2(\bar{x}+y)}^{\eta}-1))\nonumber\\
&&\qquad\qquad\qquad\qquad\le\  \snr{u(\bar{x})-u(\bar{x}+y)}^{r-2}(u(\bar{x})-u(\bar{x}+y))\\
&&\qquad\qquad\qquad\qquad\quad \ +\ (3^{q-1}+2^{q-1})\snr{\kappa\beta(\bar{x})-\kappa\beta(\bar{x}+y)+2(\snr{2(\bar{x}+y)}^{\eta}-1)}^{r-1}.\nonumber
\end{eqnarray*}
Thus, by combining the two estimates above, we get
\begin{flalign}\label{34}
I_{1}\ \le &\  \int_{\bar{x}+y\in B_{1}}\snr{u(\bar{x})-u(\bar{x}+y)}^{p-2}(u(\bar{x})-u(\bar{x}+y))K_{sp}(\bar{x},y) \dy\nonumber\\*
&\ + \int_{\bar{x}+y\in B_{1}}\hat{c}a(\bar{x},y)\snr{u(\bar{x})-u(\bar{x}+y)}^{q-2}(u(\bar{x})-u(\bar{x}+y))K_{tq}(\bar{x},y) \dy\nonumber \\*
&\ +(3^{q-1}+2^{q-1})\kappa^{p-1} \int_{\bar{x}+y\in B_{1}}\snr{\beta(\bar{x})-\beta(\bar{x}+y)}^{p-1}K_{sp}(\bar{x},y) \dy \nonumber \\*
&\ +(3^{q-1}+2^{q-1})\kappa^{q-1} \int_{\bar{x}+y\in B_{1}}\hat{c}a(\bar{x},y)\snr{\beta(\bar{x})-\beta(\bar{x}+y)}^{q-1}K_{tq}(\bar{x},y) \dy 
\end{flalign}
and
\begin{flalign}\label{35}
I_{2}\ \le&\ \int_{\bar{x}+y\not \in B_{1}}\snr{u(\bar{x})-u(\bar{x}+y)}^{p-2}(u(\bar{x})-u(\bar{x}+y))K_{sp}(\bar{x},y) \dy\nonumber\\*
&\ +\int_{\bar{x}+y\not\in B_{1}}\hat{c}a(\bar{x},y)\snr{u(\bar{x})-u(\bar{x}+y)}^{q-2}(u(\bar{x})-u(\bar{x}+y))K_{tq}(\bar{x},y) \dy\nonumber \\*
&\ +(3^{q-1}+2^{q-1})\int_{\bar{x}+y\not \in B_{1}}\snr{\kappa\beta(\bar{x})-\kappa\beta(\bar{x}+y)+2(\snr{2(\bar{x}+y)}^{\eta}-1)}^{p-1}K_{sp}(\bar{x},y) \dy \nonumber \\*[-1.5ex]
&\ +(3^{q-1}+2^{q-1})\int_{\bar{x}+y\not \in B_{1}}\hat{c}a(\bar{x},y)\snr{\kappa\beta(\bar{x})-\kappa\beta(\bar{x}+y)+2(\snr{2(\bar{x}+y)}^{\eta}-1)}^{q-1}K_{tq}(\bar{x},y) \dy, 
\end{flalign}
therefore
\begin{eqnarray*}
&&\hspace{-7mm}\mathcal{L}(u+\kappa\beta)(\bar{x})\nonumber \\*
&& \hspace{-4mm} \stackrel{\eqref{34},\eqref{35}}{\le}\mathcal{L}u(\bar{x})+(3^{q-1}+2^{q-1})\kappa^{p-1} \int_{\bar{x}+y\in B_{1}}\snr{\beta(\bar{x})-\beta(\bar{x}+y)}^{p-1}K_{sp}(\bar{x},y) \dy \nonumber \\*
&&\qquad +\ (3^{q-1}+2^{q-1})\kappa^{q-1} \int_{\bar{x}+y\in B_{1}}\hat{c}a(\bar{x},y)\snr{\beta(\bar{x})-\beta(\bar{x}+y)}^{q-1}K_{tq}(\bar{x},y) \dy \\*
&&\qquad +\ (3^{q-1}+2^{q-1})\int_{\bar{x}+y\not \in B_{1}}\snr{\kappa\beta(\bar{x})-\kappa\beta(\bar{x}+y)+2(\snr{2(\bar{x}+y)}^{\eta}-1)}^{p-1}K_{sp}(\bar{x},y) \dy \nonumber \\*
&&\qquad +\ (3^{q-1}+2^{q-1})\int_{\bar{x}+y\not \in B_{1}}\hat{c}a(\bar{x},y)\snr{\kappa\beta(\bar{x})-\kappa\beta(\bar{x}+y)+2(\snr{2(\bar{x}+y)}^{\eta}-1)}^{q-1}K_{tq}(\bar{x},y) \dy\\*
&&\stackleq{16} \sigma + (3^{q-1}+2^{q-1})\kappa^{p-1} \int_{\bar{x}+y\in B_{1}}\snr{\beta(\bar{x})-\beta(\bar{x}+y)}^{p-1}K_{sp}(\bar{x},y) \dy \nonumber \\*
&&\qquad +\ (3^{q-1}+2^{q-1})\kappa^{q-1} \int_{\bar{x}+y\in B_{1}}\hat{c}a(\bar{x},y)\snr{\beta(\bar{x})-\beta(\bar{x}+y)}^{q-1}K_{tq}(\bar{x},y) \dy \\*
&&\qquad +\ (3^{q-1}+2^{q-1})\int_{\bar{x}+y\not \in B_{1}}\snr{\kappa\beta(\bar{x})-\kappa\beta(\bar{x}+y)+2(\snr{2(\bar{x}+y)}^{\eta}-1)}^{p-1}K_{sp}(\bar{x},y) \dy \nonumber \\*
&&\qquad +\ (3^{q-1}+2^{q-1})\int_{\bar{x}+y\not \in B_{1}}\hat{c}a(\bar{x},y)\big|{\kappa\beta(\bar{x})-\kappa\beta(\bar{x}+y)+2(\snr{2(\bar{x}+y)}^{\eta}-1)\big|}^{q-1}K_{tq}(\bar{x},y) \dy.
\end{eqnarray*}
Merging the content of the above display with \eqref{14} and using \eqref{P12}, we reach the following contradiction:
\begin{eqnarray*}
\frac{\varepsilon}{\Lambda 2^{n+sp+q-1}}
\!\!&\le& \!\!2^{q-1}(1+\hat{c}M)\int_{y\not \in B_{1/4}}(\snr{8y}^{\eta}-1)^{p-1}K_{sp}(\bar{x},y)+(\snr{8y}^{\eta}-1)^{q-1}K_{tq}(\bar{x},y) \ \dy\\*
&&\!\! +\,(3^{q-1}+2^{q-1})\kappa^{p-1} \int_{\bar{x}+y\in B_{1}}\snr{\beta(\bar{x})-\beta(\bar{x}+y)}^{p-1}K_{sp}(\bar{x},y) \dy\nonumber \\*
&&\!\! +\, (3^{q-1}+2^{q-1})\kappa^{q-1} \int_{\bar{x}+y\in B_{1}}\hat{c}a(\bar{x},y)\snr{\beta(\bar{x})-\beta(\bar{x}+y)}^{q-1}K_{tq}(\bar{x},y) \dy\\*
&&\!\! +\,(3^{q-1}+2^{q-1})\int_{\bar{x}+y\not \in B_{1}}\snr{\kappa\beta(\bar{x})-\kappa\beta(\bar{x}+y)+2(\snr{2(\bar{x}+y)}^{\eta}-1)}^{p-1}K_{sp}(\bar{x},y) \dy\\*
&&\!\! +\,(3^{q-1}+2^{q-1})\int_{\bar{x}+y\not \in B_{1}}\hat{c}a(\bar{x},y)\big|{\kappa\beta(\bar{x})-\kappa\beta(\bar{x}+y)+2(\snr{2(\bar{x}+y)}^{\eta}-1)\big|}^{q-1}\\*
&& \hspace{3.9cm} \times K_{tq}(\bar{x},y) \dy\\
&\le &\!\!\frac{\varepsilon}{\Lambda 2^{n+sp+q}}.
\end{eqnarray*}
\end{proof}

\vspace{2mm}
\section{H\"older continuity}
This section is devoted to the proof of the H\"older continuity of the solutions, namely Theorem~\ref{teo_holder}. As in the local framework, at this stage, an iteration lemma will be the keypoint of the proof. However, as before, we have to handle the nonlocality of the involved operators together with the modulating function $a(\cdot,\cdot)$, and thus a certain care is required. Also, the scaling behavior of the nonlocal double phase equations has to be taking into account, precisely when dealing with the datum $f$ in order to apply the result in Lemma~\ref{lem_positivity}. Here, as expected, the assumptions on the summability exponents~\eqref{pq}-\eqref{hp_spq}
will be in force, in clear accordance with the observation in~Remark~\ref{rem_lavri}; for further clarification in this respect we refer to forthcoming Sections~\ref{sec_clarification}-\ref{sec_scala}.
On the whole, in the proof below, all the estimates proven in the previous section will take part.

\subsection{Proof of Theorem \ref{teo_holder}} 
Let $u$ be a solution of problem \eqref{problema}, with $\mathcal{L}$ as in \eqref{L}, under assumptions \eqref{K}--\eqref{energiafinita}. For a small parameter $\sigma$ chosen as in \eqref{sigma}, we first rescale $u$ by defining the map $\tilde{u}:=\lambda u$, where
\begin{equation}\label{lam0}
\lambda:=\frac{1}{2}\left(\frac{1}{\|u\|_{L^\infty}+(\|f\|_{L^\infty(B_2)}/\sigma)^{1/(p-1)}}\right).
\end{equation}
From \eqref{lam0}, it is easy to see that
\begin{flalign*}
\osc_{\mathds{R}^{n}}u\le 1
\end{flalign*}
and, according to the scaling properties of $\mathcal{L}$ discussed in the Appendix at Page~\pageref{appendix}, the function~$\tilde{u}$ satisfies 
\begin{flalign*}
\tilde{\mathcal{L}}\tilde{u}=\tilde{f} \ \ \mbox{in} \ \ B_{2},
\end{flalign*}
where $\tilde{\mathcal{L}}$ is an operator of the type of those considered in \eqref{Lc}, precisely
\begin{flalign*}
\tilde{\mathcal{L}}v(x):=&\int_{\mathds{R}^{n}}\snr{v(x)-v(x+y)}^{p-2}(v(x)-v(x+y))K_{sp}(x,y)  \dy\\*
&+\int_{\mathds{R}^{n}}\lambda^{p-q}a(x,y)\snr{v(x)-v(x+y)}^{q-2}(v(x)-v(x+y))K_{tq}(x,y)  \dy
\end{flalign*}
and
\begin{flalign*}
\tilde{f}(x):=\lambda^{p-1}f(x),
\end{flalign*}
for $x \in B_{2}$. Again, the definition in \eqref{lam0} assures that
\[
\nr{\tilde{f}}_{L^{\infty}(B_{2})}\le \frac{\sigma}{2^{p-1}}.
\]
Now, fix $x_{0}\in B_{1}$. For $i \in \mathbb{Z}$, by induction, we are going to find $b_{i}, c_{i}$ so that
\begin{flalign}\label{osc0}
b_{i}\le \tilde{u}(x)\le c_{i} \ \ \mbox{in} \ \ B_{2^{-i}}(x_{0}) \quad \mbox{and} \quad \snr{c_{i}-b_{i}}\le 2^{-i\gamma},
\end{flalign}
where $\gamma \in (0,1)$ satisfies
\begin{flalign}\label{gamma}
\frac{2-\theta}{2}\le 2^{-\gamma}\quad \mbox{and}\quad \gamma\le \eta<\min\left\{\frac{sp}{p-1},\frac{tq}{q-1}\right\}. 
\end{flalign}
The values of $\theta=\theta(\texttt{data})$ and $\eta=\eta(\texttt{data})$ are those coming from Proposition \ref{prop_cumbersome} and Lemma \ref{lem_positivity} corresponding to $\varepsilon=\snr{B_{1}}/2$. Clearly, \eqref{osc0} is satisfied for $i\le 0$ by the choice $b_{i}:=\inf_{x \in \mathds{R}^{n}}u(x)$ and $c_{i}:=b_{i}+1$. Let us assume that \eqref{osc0} holds for $i \le j$, by induction we will construct $b_{j+1}$ and $c_{j+1}$. Consider the function $\bar{u}$ defined as follows:
\begin{flalign*}
\bar{u}(x):=2^{\gamma j+1}\left(\tilde{u}(2^{-j}x+x_{0})-m\right), \ \ \mbox{for} \ \ x \in B_{1},
\end{flalign*}
where $m:=(b_{j}+c_{j})/2$. By \eqref{osc0} and \eqref{lam0}, it is evident that $\nr{\bar{u}}_{L^{\infty}(B_{1})}\le 1$ and, invoking again the behavior of the operator $\mathcal{L}$ under blow up, we obtain that $\bar{u}$ solves
\begin{flalign*}
\bar{\mathcal{L}}\bar{u}=\bar{f} \ \ \mbox{in} \ \ B_{1},
\end{flalign*}
where
\begin{flalign}\label{op0}
\bar{\mathcal{L}}v(x):=&\int_{\mathds{R}^{n}}\snr{v(x)-v(x+y)}^{p-2}(v(x)-v(x+y))\bar{K}_{sp}(x,y) \ \dy\nonumber \\
&+\int_{\mathds{R}^{n}}\bar{a}(x,y)\snr{v(x)-v(x+y)}^{q-2}(v(x)-v(x+y))\bar{K}_{tq}(x,y) \ \dy
\end{flalign}
and
\begin{flalign*}
\bar{f}(x):=\lambda^{p-1}2^{(\gamma j+1)(p-1)-jsp}f(2^{-j}x+x_{0}).
\end{flalign*}
From \eqref{lam0} and $\eqref{gamma}_{2}$ it follows that
\begin{flalign}\label{f2}
\nr{\bar{f}}_{L^{\infty}(B_{1})}\le \sigma.
\end{flalign}
Notice also that in~\eqref{op0}, the kernels $\bar{K}_{sp}(\cdot,\cdot)$ and $\bar{K}_{tq}(\cdot,\cdot)$ are given by
\begin{flalign}\label{K2}
\begin{cases}
\ \bar{K}_{sp}(x,y):=2^{-j(n+sp)}K_{sp}(2^{-j}x+x_{0},2^{-j}y)\\
\ \bar{K}_{tq}(x,y):=2^{-j(n+tq)}K_{tq}(2^{-j}x+x_{0},2^{-j}y)
\end{cases},
\end{flalign}
so that \eqref{K} is satisfied. Moreover, the new modulating coefficient is given by
\begin{flalign}\label{a0}
\bar{a}(x,y):=\lambda^{p-q}2^{-j(sp-tq)+(\gamma j+1)(p-q)}a(2^{-j}x+x_{0},2^{-j}y).
\end{flalign}
As to verify \eqref{a} for $\bar{a}(\cdot,\cdot)$, we need to analyze the behavior of $\lambda$, which is ultimately influenced only by $\sigma$. Therefore, using \eqref{sigma} we obtain that
\begin{eqnarray}\label{sigma0}
\sigma^{-}\!\!&:=&\!\!\frac{\omega_{n}2^{q-1+2sp}(2^{\eta}-1)}{sp}\nonumber\\*
&\le&\!\! \sigma \nonumber\\*
& \le&\!\! \omega_{n}2^{q+3\eta(q-1)}\max\left\{\frac{4^{sp-\eta(p-1)}}{sp-\eta(p-1)},\frac{4^{tq}-\eta(q-1)}{tq-\eta(q-1)}\right\}\,=:\,\sigma^{+},
\end{eqnarray}
where $\eta$ is by now fixed. Coupling \eqref{a0}, \eqref{a}, \eqref{lam0}, \eqref{hp_spq}, \eqref{sigma0}, and recalling that $\sigma <<1$, we can conclude that
\begin{eqnarray}\label{a0cristianaaa}
\nr{\bar{a}}_{L^{\infty}(B_{1})}\!\!&\le&\!\! 2^{q-p}M\left(\nr{u}_{L^{\infty}(\mathds{R}^{n})}+\left(\frac{\nr{f}_{L^{\infty}(B_{2})}}{\sigma}\right)^{\frac{1}{p-1}}\right)^{q-p}\nonumber \\
&\le& \!\! c(\sigma^{-})^{\frac{p-q}{p-1}}
\,=:\,\overline{M}(\texttt{data}),
\end{eqnarray}
for a constant $c$ depending only on $p,q,M,\nr{u}_{L^{\infty}(\mathds{R}^{n})},\nr{f}_{L^{\infty}(B_{2})}$, so that  \eqref{a} is satisfied as well. 

Finally, we notice that, given any $y \in \mathds{R}^{n}\setminus B_{1}$, there is an integer $\ell\ge 0$ such that $2^{\ell}\le \snr{y}\le 2^{\ell+1}$, so we have
\begin{eqnarray*}
\bar{u}(y)\!\!&=&\!\! 2^{\gamma j+1}\left(\tilde{u}(2^{-j}y+x_{0})-m\right)
\, \le\, 2^{\gamma j+1}\left(c_{j-\ell-1}-m\right)\\*
&\le &\!\! 2^{\gamma j+1}\left(c_{j-\ell-1}+b_{j}-b_{j-\ell-1}-m\right)\\*
& \le &\!\! 2^{\gamma j+1}\left(2^{-\gamma(j-\ell-1)}+b_{j}-m\right)
\, \le\, 2^{\gamma j+1}\left(2^{-\gamma(j-\ell-1)}-2^{-(\gamma j+1)}\right)\\*
&\le &\!\!2^{1+\gamma(\ell+1)}-1
\, \le\, 2\snr{2y}^{\gamma}-1
\, \le\, 2\snr{2y}^{\eta}-1,
\end{eqnarray*}
where we also used that the inductive step \eqref{osc0} holds for $i\le j$ and $\eqref{gamma}_{2}$. It is not restrictive now to suppose that $\snr{\{x \in B_{1}\colon \bar{u}(x)\le 0\}}\ge \snr{B_{1}}/2$, (otherwise we can apply the same procedure to $-\bar{u}$); collecting estimates \eqref{f2}, \eqref{K2} and \eqref{a0cristianaaa} we see that $\bar{\mathcal{L}}$, $\bar{u}$ and $\bar{f}$ satisfy all the assumptions of~Lemma~\ref{lem_positivity} corresponding to $\varepsilon=\snr{B_{1}}/2$. As a consequence, $\bar{u}(x)\le 1-\theta$ in $B_{1/2}$, for some $\theta=\theta(\texttt{data})$. Scaling back to $\tilde{u}$, this renders
\begin{eqnarray*}
\tilde{u}(x)\!\!&\le&\!\! 2^{-(\gamma j+1)}(1-\theta)+m=2^{-(\gamma j+1)}(1-\theta)+\frac{b_{j}+c_{j}}{2}\\*
&\le & \!\! b_{j}+2^{-(1+\gamma j)}(1-\theta)+2^{-(1+\gamma j)}\\*[1ex]
&\le & \!\!  b_{j}+2^{-\gamma(j+1)},
\end{eqnarray*}
where we used the definition of $m$, $\eqref{osc0}_{2}$ and $\eqref{gamma}_{1}$. Hence, the choices $b_{j+1}=b_{j}$ and $c_{j+1}=b_{j}+2^{-\gamma(j+1)}$ fix \eqref{osc0} for step $i=j+1$. Now, for any $\rr\in (0,1)$ we can find an integer~$j\ge 0$ so that $2^{-j-1}\le\rr\le 2^{-j}$ and 
\begin{eqnarray*}
\osc_{B_{\rr}(x_{0})}\tilde u\!\!&=&\!\!\sup_{x\in B_{\rr}(x_{0})}\tilde{u}(x)-\inf_{x\in B_{\rr}(x_{0})}\tilde{u}(x)\\*
&\le&\!\! b_{j-1}+2^{-\gamma j}-b_{j-1}
\ \le\  2^{\gamma}2^{-\gamma(j+1)}
\ \le\ 2^{\gamma}\rr^{\gamma}.
\end{eqnarray*}
Scaling back to $u$ in the previous estimate and recalling \eqref{lam0} and \eqref{sigma0}, we have that
\begin{eqnarray*}
\osc_{B_{\rr}(x_{0})}u
\!\!&\le&\!\! \varrho^{\gamma}2^{\gamma+1}\left(\nr{u}_{L^{\infty}(\mathds{R}^{n})}+\left(\frac{\nr{f}_{L^{\infty}(B_{2})}}{\sigma^{-}}\right)^{\frac{1}{p-1}}\right)\\*
&\le&\!\! c(\texttt{data})\left(\nr{u}_{L^{\infty}(\mathds{R}^{n})}+\nr{f}_{L^{\infty}(B_{2})}^{\frac{1}{p-1}}\right)\varrho^{\gamma}.
\end{eqnarray*}
Here, of course, $\gamma=\gamma(\texttt{data})$. By standard covering, we can finally conclude that $u \in C^{0,\gamma}(B_{1})$, as desired.
\subsection{Further clarification}\label{sec_clarification}
In the proof of Theorem \ref{teo_holder}, a crucial role is played by the estimates in Proposition \ref{prop_cumbersome}. However, in the light of the scaling features of the operator $\mathcal{L}$, it is evident that, while reducing $\sigma$ by sending $\eta\to 0^{+}$ there is some competition among the terms appearing in estimates \eqref{P10}--\eqref{P12}. Adopting the same terminology of Proposition \ref{prop_cumbersome}, from \eqref{lam0} and \eqref{a0}, we see that
\begin{equation*} 
\hat{c}\,\le\, \tilde{\lambda}^{p-q} \ \ \mbox{and} \ \ \lambda^{p-q}\sim \sigma^{-\frac{q-p}{p-1}},
\end{equation*}
where the constants implicit in ``$\sim$'' depend only on $p,q,\nr{f}_{L^{\infty}(B_{2})}, \nr{u}_{L^{\infty}(\mathds{R}^{n})}$. Thus, recalling \eqref{P10}, we look at those terms depending on $\eta$ and we can estimate
\begin{eqnarray}\label{com1}
&&\hspace{-5mm}2^{2p+q-6}\int_{x+y\not \in B_{1}}\big|{\snr{2(x+y)}^{\eta}-1}\big|^{p-1}K_{sp}(x,y)  \dy\nonumber \\*
&&\qquad+(2+\hat{c}M)2^{q-1}\int_{y \not \in B_{1/4}}\big((\snr{8y}^{\eta}-1)^{p-1}K_{sp}(x,y)+(\snr{8y}^{\eta}-1)^{q-1}K_{tq}(x,y) \big) \dy\nonumber\\
&&\qquad \le \left(2^{2p-4}\Lambda+\hat{c}M\right)2^{q-1}\int_{y \not \in B_{1/4}}\big((\snr{8y}^{\eta}-1)^{p-1}K_{sp}(x,y)+(\snr{8y}^{\eta}-1)^{q-1}K_{tq}(x,y)\big)\dy\nonumber \\*[-0.5ex]
&& \\*
&&\qquad \le c\sigma^{1-\frac{q-p}{p-1}}, \nonumber
\end{eqnarray}
since $\snr{x}$ is supposed to be less than $3/4$ in the proofs of Proposition \ref{prop_cumbersome} and of Lemma~\ref{lem_positivity}. Here, $c=c(\texttt{data})$. Condition \eqref{hp_spq} yields that the exponent of the term in the right-hand side of \eqref{com1} is strictly positive, so we can find a sufficiently small $\eta_{1}>0$ so that $c\sigma^{1-\frac{q-p}{p-1}}\le {\varepsilon}/({5\Lambda 2^{n+sp+q+1}})$. Fixed $\eta_{1}$, we can complete the bound on those terms appearing on the left-hand side of \eqref{P10} as follows
\begin{flalign}\label{com2}
2^{q-2}\kappa_{1}^{p-1}&\int_{x+y\in B_{1}}\snr{\beta(x)-\beta(x+y)}^{p-1}K_{sp}(x,y) \ \dy\nonumber \\*
&+2^{q-2}\kappa_{1}^{q-1}\int_{x+y\in B_{1}}\hat{c}a(x,y)\snr{\beta(x)-\beta(x+y)}^{q-1}K_{tq}(x,y) \ \dy\nonumber \\*
&+2^{p+q-4}\kappa_{1}^{p-1}\int_{x+y\not \in B_{1}}\snr{\beta(x)-\beta(x,y)}^{p-1}K_{sp}(x,y) \ \dy\nonumber \\*
&+2^{2q-4}\kappa_{1}^{q-1}\int_{x+y\not \in B_{1}}\hat{c}a(x,y)\snr{\beta(x)-\beta(x+y)}^{q-1}K_{tq}(x,y) \ \dy\nonumber \\
\le & \ c\left(\kappa_{1}^{p-1}+\sigma^{-\frac{q-p}{p-1}}\kappa_{1}^{q-1}\right)
\, \le\, \frac{\varepsilon}{5\Lambda 2^{n+sp+q+1}},
\end{flalign}
provided that $\kappa_{1}\le \left(\frac{\sigma}{2c}\right)^{\frac{1}{p-1}}$, with $c=c(\texttt{data})$. Then, we can easily deduce inequality~\eqref{P10} by merging the contents of~\eqref{com1} and~\eqref{com2}. In a totally similar way we can manipulate the quantities appearing in~\eqref{P11} and in~\eqref{P12} as to verify those bounds and, at the same time, avoid any competition among the parameters involved.

Now, we let $\eta_{2}$,$\kappa_{2}$ and $\eta_{3}$,$\kappa_{3}$ be those values of~$\eta$ and of~$\kappa$ for which~\eqref{P11} and~\eqref{P12} hold respectively, and finally set $\eta:=\min\{\eta_{1},\eta_{2},\eta_{3}\}$ and $\kappa:=\min\{\kappa_{1},\kappa_{2},\kappa_{3}\}$. In this way, \eqref{P10}--\eqref{P12} are simultaneously satisfied. Notice that~$\eta=\eta(\texttt{data},\varepsilon)$ and~$\kappa=\kappa(\texttt{data},\varepsilon)$.

\vspace{3mm}
\section*{Appendix}\label{appendix}
We complete this paper by discussing some intrinsic characteristics of the class of operators under investigation. 
\subsection{Scaling properties of the nonlocal double phase equations}\label{sec_scala}
Let us analyze the structure of operator~$\mathcal{L}$ by investigating some general scaling properties. Under assumption~\eqref{energiafinita}, let $u\in L^{\infty}(\mathds{R}^{n})$ be a viscosity solution to problem~\eqref{problema} in the sense of Definition \ref{def_viscosity}. 
We rescale and blow~$u$ around a point~$x_{0}\in B_{1}$ as follows. For~$\lambda, \mu>0$ and~$x\in B_{1}$, we define the map~$u_{\mu,x_{0}}^{(\lambda)}(x):=\lambda u(\mu x+x_{0})$. Such a function satisfies 
\[
\hat{\mathcal{L}}u^{(\lambda)}_{\mu,x_{0}}(x):=\hat{f}(x) \ \ \mbox{in} \ \ B_{1},
\] 
where
\begin{eqnarray}\label{sc}
\hat{\mathcal{L}}v(x)\!\!&:=&\!\!\int_{\mathds{R}^{n}}\snr{v(x)-v(x+y)}^{p-2}(v(x)-v(x+y))\hat{K}_{sp}(x,y) \dy\nonumber \\
&&\!\!+\,\int_{\mathds{R}^{n}}\hat{a}(x,y)\snr{v(x)-v(x+y)}^{q-2}(v(x)-v(x+y))\hat{K}_{tq}(x,y)  \dy
\end{eqnarray}
and
\begin{flalign}\label{newf}
\hat{f}(x):=\lambda^{p-1}\mu^{sp}f(\mu x+x_{0}).
\end{flalign}
The modulating coefficient and the kernels appearing in \eqref{sc} are defined as
\begin{flalign}\label{newa}
\hat{a}(x,y):=\lambda^{p-q}\mu^{sp-tq}a(\mu x+x_{0},\mu y)
\end{flalign}
and
\begin{flalign}\label{newks}
\begin{cases}
\ \hat{K}_{sp}(x,y):=\mu^{n+sp}K_{sp}(\mu x+x_{0},\mu y)\\[0.8ex]
\ \hat{K}_{tq}(x,y):=\mu^{n+tq}K_{tq}(\mu x+x_{0},\mu y)
\end{cases},
\end{flalign}
respectively. From \eqref{newks}, it immediately follows that $\hat{K}_{sp}$ and $\hat{K}_{tq}$ satisfy \eqref{K}. Moreover, \eqref{newa} yields \eqref{a} via replacing $M$ with $\hat{M}:=\lambda^{p-q}\mu^{sp-tq}M$ and, by \eqref{newf} it follows that $\nr{\hat{f}}_{L^{\infty}(B_{1})}\le\lambda^{p-1}\mu^{sp-tq}\nr{f}_{L^{\infty}(B_{1})}$.
\subsection{Some useful inequalities}
We report some elementary algebraic inequalities whose proofs are essentially contained in the Appendix in~\cite{Lin16}; see in particular Lemma~3 and Lemma~4 there. We refer also to Section 2 in~\cite{KKP16} where similar inequalities do appear. All of them are very useful in estimating the $p$-fractional Sobolev seminorms, and they were needed in the whole paper. For the sake of completeness, here below we address the plain modifications in order to extend them to the fractional $(p,q)$-growth case.
\begin{lemma}\label{revL1}
Let $r\ge 2$, $r\in \{p,q\}$. Then
\begin{flalign*}
\left |\snr{a+b}^{r-2}(a+b)-\snr{a}^{r-2}a \right |\le (r-1)\snr{b}(\snr{a}+\snr{b})^{r-2},
\end{flalign*}
for all $a,b\in \mathds{R}$.
\end{lemma}
\begin{proof}
The proof is contained in \cite[Lemma 2]{Lin16}.
\end{proof}
\begin{lemma}\label{lem_superlinear}
Let $q\ge p\ge 2$ and let $a,b\in \mathds{R}$ such that $a+b\ge 0$. Then, for $r\in\{p,q\}$ there holds 
\begin{flalign*}
\snr{a+b}^{r-2}(a+b)
\, \le\, 2^{q-2}(\snr{a}^{r-2}a+\snr{b}^{r-2}b).
\end{flalign*}
\end{lemma}
\begin{proof}
Firstly, {notice that, $a+b\ge 0$ implies $\snr{a}^{r-2}a+\snr{b}^{r-2}b\ge 0$ and,} by homogeneity, the statement is equivalent to $\snr{1+\tau}^{r-2}(1+\tau)\le 2^{r-2}(1+\snr{\tau}^{r-2}\tau)$ with $\tau\ge -1$ and $r\in\{p,q\}$. For this, it will suffice to study the asymptotics of the function $\dys \tau\mapsto f_{r}(\tau):=\frac{\snr{1+\tau}^{r-2}(1+\tau)}{1+\snr{\tau}^{r-2}\tau}$ for $\tau\ge -1$. It is easy to check that $f_{r}(\tau)\le 2^{r-2}\le 2^{q-2}$, and this will give the desired inequality.
\end{proof}
\begin{lemma}\label{lem_singular}
Let either $p\in (1,2)$ or $q\in(1,2)$, $q\ge p$. Then, for $r\in\{p,q\}$ there holds 
\begin{flalign*}
\left |\snr{a+b}^{r-2}(a+b)-\snr{a}^{r-2}a \right |
\, \le\, (3^{q-1}+2^{q-1})\snr{b}^{r-1}.
\end{flalign*}
\end{lemma}
\begin{proof}
One can basically follow the proof of Lemma 3 in~\cite{Lin16}, by observing that $q\ge p$ plainly yields  $3^{p-1}+2^{p-1}<3^{q-1}+2^{q-1}$.
\end{proof}
Important consequences of Lemma \ref{revL1} are the following results for $C^{2}$-regular functions.
\begin{lemma}\label{revL}
Let $\varphi\in C^{2}(\mathds{R}^{n})$. If $p\ge 2$. Then,  there holds
\begin{flalign}\label{rev3}
&\snr{\varphi(x)-\varphi(x+y)}^{p-2}(\varphi(x)-\varphi(x+y))+\snr{\varphi(x)-\varphi(x-y)}^{p-2}(\varphi(x)-\varphi(x-y))\nonumber \\
&\quad \le c\snr{y}^{p} \ \ \mbox{for all} \ \ x,y\in \mathds{R}^{n},
\end{flalign}
with $c\equiv c(n,p,\nr{\varphi}_{C^{2}})$.

 Let $1<p<2$. Then,  there holds
\begin{flalign}\label{rev10}
&\snr{\varphi(x)-\varphi(x+y)}^{p-2}(\varphi(x)-\varphi(x+y))+\snr{\varphi(x)-\varphi(x-y)}^{p-2}(\varphi(x)-\varphi(x-y))\nonumber \\
&\quad \le c\snr{y}^{p-1} \ \ \mbox{for all} \ \ x,y\in \mathds{R}^{n},
\end{flalign}
for $c=c(n,p,\nr{\varphi}_{C^{1}})$. 

Moreover, if $a(\cdot,\cdot)$ satisfies~\eqref{a}, then for every~$q>1$ we have
\begin{eqnarray}\label{rev11}
&&a(x,y)\snr{\varphi(x)-\varphi(x+y)}^{q-2}(\varphi(x)-\varphi(x+y))\nonumber \\
&&\qquad\qquad \ \ \ + \, a(x,-y)\snr{\varphi(x)-\varphi(x-y)}^{q-2}(\varphi(x)-\varphi(x-y))\nonumber \\
&&\qquad\qquad \le c\snr{y}^{q-1},
\end{eqnarray}
where $c=c(n,q,\nr{a}_{L^\infty},\nr{\varphi}_{C^{1}})$. 

If $a$ belongs to~$C^{0,\alpha}(\mathds{R}^{n}\times \mathds{R}^{n})$ for some $\alpha \in (0,1]$ and $q\ge 2$, then
\begin{eqnarray}\label{rev4}
&&a(x,y)\snr{\varphi(x)-\varphi(x+y)}^{q-2}(\varphi(x)-\varphi(x+y))\nonumber \\
&&\qquad\qquad \ \ \ + \, a(x,-y)\snr{\varphi(x)-\varphi(x-y)}^{q-2}(\varphi(x)-\varphi(x-y))\nonumber \\
&&\qquad \qquad \le c(\snr{y}^{q}+\snr{y}^{q-1+\alpha}), 
\end{eqnarray}
with $c=c(n,q,\nr{a}_{L^\infty},[a]_{C^{0,\alpha}},\nr{\varphi}_{C^{2}})$.

Finally, if $a(x,y)=a(x,-y)$ for all $x,y\in \mathbb{R}^{n}$ and $q\geq 2$, then
\begin{eqnarray}\label{rev30}
&&a(x,y)\snr{\varphi(x)-\varphi(x+y)}^{q-2}(\varphi(x)-\varphi(x+y))\nonumber \\
&&\qquad\qquad \ \ \ + \, a(x,-y)\snr{\varphi(x)-\varphi(x-y)}^{q-2}(\varphi(x)-\varphi(x-y))\nonumber \\
&&\qquad \qquad \le c\snr{y}^{q},
\end{eqnarray}
holds true for $c=c(n,q,\nr{a}_{L^\infty},\nr{\varphi}_{C^{2}})$.
\end{lemma}
\begin{proof}
Set
\begin{flalign*} &A:=-(\varphi(x)-\varphi(x+y)),\\ &B:=(\varphi(x)-\varphi(x+y))+(\varphi(x)-\varphi(x+y)). \end{flalign*}
In these terms, we can apply Lemma \ref{revL1} to get
\begin{eqnarray*}
&&\snr{\varphi(x)-\varphi(x+y)}^{p-2}(\varphi(x)-\varphi(x+y))+\snr{\varphi(x)-\varphi(x-y)}^{p-2}(\varphi(x)-\varphi(x-y))\nonumber \\
&&\qquad\qquad\qquad= \ \snr{A+B}^{p-2}(A+B)-\snr{A}^{p-2}A\\
&&\qquad\qquad\qquad\le \ (p-1)\snr{B}(\snr{A}+\snr{B})^{p-1}\\
&&\qquad\qquad\qquad\le \ c\snr{y}^{p},
\end{eqnarray*}
for all $x,y\in \mathds{R}^{n}$. Here $c=c(n,p,\nr{\varphi}_{C^{2}})$, which is \eqref{rev3}. For \eqref{rev10}, we only need Lagrange's Theorem to obtain
\begin{eqnarray*}
&&\snr{\varphi(x)-\varphi(x+y)}^{p-2}(\varphi(x)-\varphi(x+y))+\snr{\varphi(x)-\varphi(x-y)}^{p-2}(\varphi(x)-\varphi(x-y))\nonumber \\
&&\qquad\qquad\qquad \quad \le c\snr{y}^{p-1},
\end{eqnarray*}
with $c=c(n,p,\nr{\varphi}_{C^{1}})$. On the other hand, due to the presence of $a(\cdot,\cdot)$, which is a priori only bounded, we cannot apply the same procedure to the $q$-part of our operator, therefore, under \eqref{a} we can only invoke Lagrange's Theorem again to get 
\begin{eqnarray*}
&&a(x,y)\snr{\varphi(x)-\varphi(x+y)}^{q-2}(\varphi(x)-\varphi(x+y))\nonumber\\
&&\qquad\qquad \ \ \ + \, a(x,-y)\snr{\varphi(x)-\varphi(x-y)}^{q-2}(\varphi(x)-\varphi(x-y))\nonumber \\
&&\qquad \qquad \le c\snr{y}^{q-1} \ \ \mbox{for all} \ \ x,y\in \mathds{R}^{n},
\end{eqnarray*}
with $c=c(n,q,\nr{a}_{L^\infty},\nr{\varphi}_{C^{1}})$. Nonetheless, if we assume extra regularity for $a(\cdot,\cdot)$, we can exploit again Lemma \ref{revL1} to obtain
\begin{eqnarray*}
&&a(x,y)\snr{\varphi(x)-\varphi(x+y)}^{q-2}(\varphi(x)-\varphi(x+y))\nonumber\\
&&\qquad\qquad \ \ \, +\, a(x,-y)\snr{\varphi(x)-\varphi(x-y)}^{q-2}(\varphi(x)-\varphi(x-y))\nonumber \\ 
&&\qquad \qquad \ \ \, \pm \, a(x,y)\snr{\varphi(x)-\varphi(x-y)}^{q-2}(\varphi(x)-\varphi(x-y))\nonumber \\ 
&&\qquad\qquad \le \, a(x,y)\left(\snr{A+B}^{q-2}(A+B)-\snr{A}^{q-2}B\right)\nonumber \\ 
&&\qquad \qquad \, \ \ +\, 2[a]_{C^{0,\alpha}}\snr{y}^{\alpha}\snr{\varphi(x)-\varphi(x-y)}^{q-1}
\, \le\,  c(\snr{y}^{q}+\snr{y}^{q-1+\alpha}),
\end{eqnarray*}
where $c=c(n,q,\nr{a}_{L^\infty},[a]_{C^{0,\alpha}},\nr{\varphi}_{C^{2}})$. Notice that, when $\snr{y}\le 1$, the last term in~\eqref{rev4} can be bounded by $c\snr{y}^{q-1+\alpha}$. 

Finally, we conclude the proof by stressing that the procedure employed to obtain~\eqref{rev3} can be repeated verbatim when $q\ge2$ provided that $a(x,y)=a(x,-y)$ for all $x,y \in \mathbb{R}^{n}$. In such a case, $a(\cdot
,\cdot)\ge 0$ does reduce to a multiplicative factor allowing again the application of Lemma \ref{revL1}; we have
\begin{eqnarray*}
&&a(x,y)\snr{\varphi(x)-\varphi(x+y)}^{q-2}(\varphi(x)-\varphi(x+y))\nonumber \\
&&\qquad\qquad \ \ \, +\, a(x,-y)\snr{\varphi(x)-\varphi(x-y)}^{q-2}(\varphi(x)-\varphi(x-y))\nonumber\\
&&\qquad\qquad=a(x,y)\left(\snr{\varphi(x)-\varphi(x+y)}^{q-2}(\varphi(x)-\varphi(x+y))\right.\nonumber \\
&&\qquad \qquad \ \ \, \left.+\,\snr{\varphi(x)-\varphi(x-y)}^{q-2}(\varphi(x)-\varphi(x-y))\right)
\ \le\ c\snr{y}^{q},
\end{eqnarray*}
with $c=c(n,q,\nr{a}_{\infty},\nr{\varphi}_{C^{2}})$, which is \eqref{rev30}.
\end{proof}

\begin{lemma}\label{revL2}
Let $B_{\rr}\subset \mathds{R}^{n}$ be any ball centered in the origin with radius $\rr\le 1$ and $\varphi \in C^{2}(\mathds{R}^{n})$. Then, if $p\ge 2$ we have
\begin{flalign}\label{rev5}
\int_{B_{\rr}}&\left |\snr{\varphi(x)-\varphi(x+y)}^{p-2}(\varphi(x)-\varphi(x+y))\right.\nonumber \\
&\left.+\snr{\varphi(x)-\varphi(x-y)}^{p-2}(\varphi(x)-\varphi(x-y)) \right |K_{sp}(x,y) \dy\le c<\infty,
\end{flalign}
for $c=c(n,\Lambda,p,\nr{\varphi}_{C^{2}(\mathds{R}^{n})})$. If ${1}/{(1-s)}<p<2$, there holds 
\begin{flalign}\label{rev9}
\int_{B_{\rr}}&\left |\snr{\varphi(x)-\varphi(x+y)}^{p-2}(\varphi(x)-\varphi(x+y))\right.\nonumber \\
&\left.+\snr{\varphi(x)-\varphi(x-y)}^{p-2}(\varphi(x)-\varphi(x-y)) \right |K_{sp}(x,y)  \dy\le c<\infty,
\end{flalign}
with $c=c(n,\Lambda,p,\nr{\varphi}_{C^{1}(\mathds{R}^{n})})$. Moreover, under assumption \eqref{a} we have
\begin{flalign}\label{rev6}
\int_{B_{\rr}}&\left |a(x,y)\snr{\varphi(x)-\varphi(x+y)}^{q-2}(\varphi(x)-\varphi(x+y))\right.\nonumber \\
&\left.+a(x,-y)\snr{\varphi(x)-\varphi(x-y)}^{q-2}(\varphi(x)-\varphi(x-y)) \right |K_{tq}(x,y)  \dy\le c<\infty,
\end{flalign}
for $c=c(n,\Lambda,q,\nr{a}_{L^\infty},\nr{\varphi}_{C^{1}})$.

If $a\in C^{0,\alpha}(\mathds{R}^{n}\times \mathds{R}^{n})$ for some $\alpha \in (0,1]$ and $q\ge \max\big\{2,\, {(1-\alpha)}{(1-t)}\big\}$, we have
\begin{flalign}\label{rev8}
\int_{B_{\rr}}&\left |a(x,y)\snr{\varphi(x)-\varphi(x+y)}^{q-2}(\varphi(x)-\varphi(x+y))\right.\nonumber \\
&\left.+a(x,-y)\snr{\varphi(x)-\varphi(x-y)}^{q-2}(\varphi(x)-\varphi(x-y)) \right |K_{tq}(x,y)  \dy\le c<\infty,
\end{flalign}
is satisfied. Here $c=c(n,\Lambda,q,\nr{a}_{L^\infty},[a]_{C^{0,\alpha}},\nr{\varphi}_{C^{2}})$.

Finally, if $a\in L^{\infty}(\mathbb{R}^{n}\times \mathbb{R}^{n})$ symmetric in the sense that $a(x,y)=a(x,-y)$ for all $x,y\in \mathbb{R}^{n}\times \mathbb{R}^{n}$, and $q\geq 2$, then
\begin{flalign}\label{rev31}
\int_{B_{\rr}}&\left |a(x,y)\snr{\varphi(x)-\varphi(x+y)}^{q-2}(\varphi(x)-\varphi(x+y))\right.\nonumber \\
&\left.+a(x,-y)\snr{\varphi(x)-\varphi(x-y)}^{q-2}(\varphi(x)-\varphi(x-y)) \right |K_{tq}(x,y) \ \dy\le c<\infty
\end{flalign}
holds true with $c=c(n,\Lambda,q,\nr{a}_{\infty},\nr{\varphi}_{C^{2}})$.

\end{lemma}
\begin{proof}
Let $B_{\rr}\subset \mathds{R}^{n}$ be any ball centered in the origin with radius $\rr\le 1$. If $p\ge 2$, recalling %$\eqref{pq}_{1}$, 
\eqref{K} and \eqref{rev3},
\begin{flalign*}
\int_{B_{\rr}}&\left |\snr{\varphi(x)-\varphi(x+y)}^{p-2}(\varphi(x)-\varphi(x+y))\right.\nonumber \\
&\left.+\snr{\varphi(x)-\varphi(x-y)}^{p-2}(\varphi(x)-\varphi(x-y)) \right |K_{sp}(x,y) \dy\nonumber \\
&\quad \le c\int_{B_{\rr}}\snr{y}^{p-n-sp}  \dy<c<\infty,
\end{flalign*}
where $c=c(n,\Lambda,p,\nr{\varphi}_{C^{2}(\mathds{R}^{n})})$. If $p<2$, by \eqref{rev10} and $\eqref{pq}_{1}$ we have
\begin{flalign*}
\int_{B_{\rr}}&\left |\snr{\varphi(x)-\varphi(x+y)}^{p-2}(\varphi(x)-\varphi(x+y))\right.\nonumber \\
&\left.+\snr{\varphi(x)-\varphi(x-y)}^{p-2}(\varphi(x)-\varphi(x-y)) \right |K_{sp}(x,y) \dy\nonumber \\
&\quad \le c\int_{B_{\rr}}\snr{y}^{p-1-n-sp}  \dy<c<\infty,
\end{flalign*}
with $c=c(n,\Lambda,p,\nr{\varphi}_{C^{1}})$. Furthermore, if we assume \eqref{a}, by $\eqref{pq}_{2}$, \eqref{K} and \eqref{rev11} it follows that
\begin{flalign*}
\int_{B_{\rr}}&\left |a(x,y)\snr{\varphi(x)-\varphi(x+y)}^{q-2}(\varphi(x)-\varphi(x+y))\right.\nonumber \\
&\left.+a(x,-y)\snr{\varphi(x)-\varphi(x-y)}^{q-2}(\varphi(x)-\varphi(x-y)) \right |K_{tq}(x,y)  \dy\nonumber \\
&\quad \le c\int_{B_{\rr}}\snr{y}^{q-1-n-tq}  \dy\, \le\, c,
\end{flalign*}
for $c=c(n,\Lambda,q,\nr{a}_{L^\infty},\nr{\varphi}_{C^{1}})$. Clearly, in this case it does not matter if $q<2$ or $q\ge 2$. 

Finally, if $q\ge 2$ and in addition $a$ belongs to $C^{0,\alpha}(\mathds{R}^{n}\times \mathds{R}^{n})$, from \eqref{rev4} we get
\begin{flalign*}
\int_{B_{\rr}}&\left |a(x,y)\snr{\varphi(x)-\varphi(x+y)}^{q-2}(\varphi(x)-\varphi(x+y))\right.\nonumber \\
&\left.+a(x,-y)\snr{\varphi(x)-\varphi(x-y)}^{q-2}(\varphi(x)-\varphi(x-y)) \right |K_{tq}(x,y) \ \dy\nonumber \\
&\quad \le c\int_{B_{\rr}}\snr{y}^{q-1+\alpha}  \dy
\, \le\, c,
\end{flalign*}
provided that also $q>{(1-\alpha)}{(1-t)}$. Here, $c=c(n,\Lambda,q,\nr{a}_{L^\infty},[a]_{C^{0,\alpha}},\nr{\varphi}_{C^{2}})$.

Finally, concerning the estimate in~\eqref{rev31}, by recalling \eqref{rev30} and \eqref{K}, we can conclude that
\begin{flalign*}
\int_{B_{\rr}}&\left |a(x,y)\snr{\varphi(x)-\varphi(x+y)}^{q-2}(\varphi(x)-\varphi(x+y))\right.\nonumber \\
&\left.+a(x,-y)\snr{\varphi(x)-\varphi(x-y)}^{q-2}(\varphi(x)-\varphi(x-y)) \right |K_{tq}(x,y) \ \dy\nonumber \\
&\quad \le c\int_{B_{\rr}}\snr{y}^{q} \ \dy\le c,
\end{flalign*}
for $c=c(n,\Lambda,q,\nr{a}_{\infty},\nr{\varphi}_{C^{2}})$.
\end{proof}

\vspace{2mm}

\end{document}